\documentclass[10pt, reqno]{amsart}

\usepackage[numbers,sort&compress]{natbib}
\usepackage{amssymb,amsmath,amsfonts, amsthm}
\usepackage{latexsym}
\usepackage{xcolor}
\usepackage[mathscr]{eucal}
\usepackage{comment}

\usepackage[linkcolor=blue, citecolor=black]{hyperref}
\hypersetup{colorlinks=true, urlcolor=blue}
\usepackage[nameinlink]{cleveref}

\numberwithin{equation}{section}

\theoremstyle{definition}
\newtheorem{definition}{Definition}[section]

\theoremstyle{remark}
\newtheorem{remark}[definition]{Remark}
\theoremstyle{plain}
\newtheorem{theorem}[definition]{Theorem}
\newtheorem{lemma}[definition]{Lemma}
\newtheorem{proposition}[definition]{Proposition}

\newtheorem{result}[definition]{Result}

\newtheorem{question}[definition]{Question}

\newcommand{\eps}{\varepsilon}

\newcommand{\al}{\alpha}

\newcommand{\tht}{\theta}

\newcommand{\rank}{\operatorname{rank}}

\newcommand{\OM}{\Omega}

\newcommand{\smoo}{\mathcal{C}}
\newcommand{\hol}{\mathcal{O}}
\newcommand{\poly}{\mathscr{P}}

\newcommand{\rl}{{\sf Re}}
\newcommand{\imag}{{\sf Im}}
\newcommand{\impl}{\Longrightarrow}

\newcommand\hull[1]{\widehat{#1}}

\newcommand{\delbydel}[2]{\frac{\partial #1}{\partial #2}}

\newcommand{\cplx}{\mathbb{C}}

\begin{document}
	\title[Polynomial approximation on certain sets]{Polynomial Convexity and Polynomial approximations of certain sets in $\mathbb{C}^{2n}$ with non-isolated CR-singularities}
	\author{Golam Mostafa Mondal}
	\address{Department of Mathematics and Statistics, Indian Institute of Science Education and Research Kolkata,
		Mohanpur -- 741 246}
	\email{golammostafaa@gmail.com}

	\begin{abstract}
		In this paper, we first consider the graph of $(F_1,F_{2},\cdots,F_{n})$ on $\overline{\mathbb{D}}^{n},$ where $F_{j}(z)=\bar{z}^{m_{j}}_{j}+R_{j}(z),j=1,2,\cdots,n,$ which has non-isolated CR-singularities if $m_{j}>1$ for some $j\in\{1,2,\cdots,n\}.$
		We show that under certain condition on $R_{j},$ the graph is polynomially convex and holomorphic polynomials on the graph approximates all continuous functions. We also show that there exists an open polydisc $D$ centred at the origin such that the set $\{(z^{m_{1}}_{1},\cdots, z^{m_{n}}_{n}, \bar{z_1}^{m_{n+1}} + R_{1}(z),\cdots, \bar{z_{n}}^{m_{2n}} + R_{n}(z)):z\in \overline{D},m_{j}\in \mathbb{N}, j=1,\cdots,2n\}$ is polynomially convex; and if $\gcd(m_{j},m_{k})=1~~\forall j\not=k,$ the algebra generated by the functions $z^{m_{1}}_{1},\cdots, z^{m_{n}}_{n}, \bar{z_1}^{m_{n+1}} + R_{1},\cdots, \bar{z_{n}}^{m_{2n}} + R_{n}$ is dense in $\smoo(\overline{D}).$ We prove an analogue of Minsker's theorem over the closed unit polydisc, i.e, if $\gcd(m_{j},m_{k})=1~~\forall j\not=k,$ the algebra  $[z^{m_{1}}_{1},\cdots, z^{m_{n}}_{n}, \bar{z_1}^{m_{n+1}},\cdots , \bar{z_{n}}^{m_{2n}};{\overline{\mathbb{D}}^{n}} ]=\smoo(\overline{\mathbb{D}}^{n}).$ In the process of proving the above results, we also studied the polynomial convexity and approximation of certain graphs.
	\end{abstract}

	\maketitle
	\section{Introduction and statements of the results}\label{S:intro}
	Let $D$ be an open polydisc in $\cplx^n$ with center at the origin, and by $\smoo(\overline{D}),$ we denote the set of all continuous complex valued functions on $\overline{D}.$ For $f_1,f_2,\cdots,f_N\in \smoo(\overline{D}),$ we denote by $[f_1,f_2,\cdots,f_N;\overline{D}]$ the uniform algebra generated by $f_1, f_2, \cdots, f_N$ on $\overline{D}.$ In this article, we report our investigation to the following question:
	\begin{question}\label{Q:Unfrm alg}
		Let $\nu_{1},\cdots ,\nu_{n}$ be positive integers. Under what conditions on $ f_1,\cdots,f_n,$ can one conclude
		\begin{align*}
			[z^{\nu_{1}}_{1},\cdots, z^{\nu_{n}}_{n}, f_1,f_2,\cdots , f_n; \overline{D}] = \smoo(\overline{D})?
		\end{align*}
	\end{question}
	In this discussion we first focus on the case  when $\nu_j=1~~\forall j\in \{1,2,\cdots,n\}.$ This problem seems relatively easier because the underlined set is the graph of $(f_{1},\cdots,f_{n})$ over the closed polydisc $\overline{D}.$ We present a brief literature survey on this. The problem is quite well studied for $n=1.$ The following result by Mergelyan \cite{MAR54} played a vital role in creating some interest in this question.
	
	\begin{result}[Mergelyan]\label{R:Mergyln}
		Let $D$ be an open disc in $\cplx$ and let $f$ be a continuous real-valued
		function on $\overline{D}$. If for each $a$ in $\overline{D},$ $f^{-1}\left(f(a)\right)$ has no interior and does not separate $\cplx$, then $[z,f; \overline{D}] = \smoo(\overline{D}),$ where $[z,f; \overline{D}]$ denotes
		the algebra generated by the functions $z$ and $f$ with complex coefficients.
	\end{result}
	\noindent We now state a couple of result due to Wermer \cite{W1}.
	
	\begin{result}[Wermer]\label{R:W164}
		Let $D$ be an open unit disc in $\cplx$ with center at the origin. If $f(z)=\bar{z}+R(z),$ where 
		\begin{align}
			\left|R(z)-R(a)\right|<|z-a|
		\end{align}
		for all $a,z$ in $\overline{D}$ with $a\ne z,$ then $[z,f;\overline{D}]=\smoo(\overline{D}).$
	\end{result}
	\begin{result}[Wermer]\label{R:W264}
		Fix $\delta_{0}>0.$ Let $g$ be function defined in the disc $\{z\in \cplx:|z|<\delta_{0}\}$ and have continuous partial derivatives up to the second order there. Assume 
		\begin{align}\label{Eq:Totrl}
			\delbydel{g}{\bar{z}}(0)\ne 0.
		\end{align}
		Then there exist $\delta,$ $0<\delta<\delta_{0},$ such that for $D=\{z:|z|\leq\delta\},$ $[z,g;D]=C(D).$	
	\end{result}	
	
	\noindent There are several results due to Preskenis \cite{PRE1,PRE2,PRE3}, O'Farrell, Preskenis and Walsh \cite{OP1}, (see Sanabria \cite{sangar} for a nice survey) that generalize Wermer's results.
	\par Before going further, we discuss the relation between the above mentioned approximation problem with polynomial convexity, which is a fundamental notion in several complex variables. Let $K$ be a compact subset of $\cplx^n.$ The \textit{polynomial convex hull} of $K$ is denoted by $\hull K$ and defined by $\hull K:=\left\{\al\in\cplx^n: |p(\al)|\le\max_{K}|p|~~\forall p\in \cplx[z_1,z_2,\cdots,z_n]\right\}.$ Clearly, $K\subset\hull K.$ We say $K$ is \textit{polynomially convex} if $\hull K= K.$ In $\cplx,$ $\hat{K}=K$ if and only if $\cplx\setminus K$ is connected. Any convex compact subset of $\cplx^{n}$ is polynomially convex. In general, it is difficult to deduce if a given compact subset in $\cplx^{n}$ is polynomially convex. A closed subset $E$ is \textit{locally polynomially convex} at $p\in E$ if there exists $r>0$ such that $E\cap \overline{B(p,r)}$ is polynomially convex. Let $K$ be a compact set in $\cplx^n$ and let $\smoo(K)$ be the class of  of all continuous complex-valued functions on $K$. By $\poly(K)$, we denote the space of all those functions on $K$ which are uniform limits of polynomials in $z_1,z_2,\cdots,z_n$. One of the fundamental question in the theory of uniform algebras is to characterize compact subset of $\cplx^n$ for which 
	\begin{equation}\label{E:approx}
		\poly(K)=\smoo(K).
	\end{equation}
	
	\noindent From the theory of commutative Banach algebras (see \cite{Gamelin} for details),  we notice that
	\[
	\poly(K)=\smoo(K)\impl \hull{K}=K. 
	\]
	Therefore, polynomial convexity is a necessary condition for all compacts $K$ of $\cplx^n$ having property (\ref{E:approx}). Lavrentiev \cite{Lv} showed that for $K\subset \cplx,$ $\poly(K)=\smoo(K)$ if and only if $\hull{K}=K$ and $int(K)=\emptyset.$ In higher dimension, no such characterization is known. The condition $\cplx\setminus K$ connected generalizes to polynomial convexity of $K.$ A generalization of $int(K)=\emptyset$ condition is that the compact $K$ is totally real except a small set of points. Recall that a $C^{1}$-smooth submanifold $M$ of $\cplx^{n}$ is said to be \textit{totally real} at $p\in M$ if $T_{p}M\cap iT_{p}M=\{0\},$ where the tangent space $T_{p}M$ is viewed as a real linear subspace of $\cplx^{n}.$ Otherwise, we say that $M$ has complex tangent at $p.$ The manifold $M$ is said to be totally real if it is totally real at all points of $M$ and a subset $K$ of $\cplx^n$ is said to be totally real if it is locally contained in a totally real manifold. \Cref{R:W264} due to Wermer says that a totally real submanifold of $\cplx^{2}$ is locally polynomially convex. H$\ddot{o}$rmander and Wermer \cite{HoW68} generalized this result for smooth totally real submanifold of $\cplx^{n}.$ Smoothness is reduced to $C^{1}$ due to Harvey and Wells \cite{HarWel71,HarWel72}. For polynomially convex set $K$, there are several papers, for instance see \cite{AIW1, AIW2, OPW, St1, W1},
	that describe situations when \eqref{E:approx} holds. For $n>1,$ the most general result known in this direction is the following due to O'Farrell, Preskenis and Walsh \cite{OPW}:

	\begin{result}[O'Farrell, Preskenis and Walsh]\label{R:OPW}
		Let K be a compact polynomially convex subset of $\cplx^n$. Assume that $E$ is a closed subset of $K$ such that
		$K\setminus E$ is locally contained in totally-real manifold. Then
		\[
		\poly(K)=\{f\in \smoo(K): f|_E\in \poly(E)\}.
		\]
	\end{result}
	From \Cref{R:OPW}, we can say that any compact polynomially convex subset of a totally real submanifold of $\cplx^n$ enjoys property (\ref{E:approx}). Therefore, to answer the approximation question mentioned in the beginning or (\ref{E:approx}), one requires to answer certain polynomial convexity question. Gorai \cite{SG17}, Chi \cite {KPC}, Zajec \cite{MZ} gave some results regarding the polynomial convexity of a compact that lies in a totally real submanifold of $\cplx^{n}.$ When the graph has certain isolated points with complex tangents (CR-singularity), the question of local polynomial convexity near the CR-singularity becomes more complicated. The study has been initiated with Bishop's \cite{Bishop65} work of attaching analytic discs. Forstneri\v c-Stout \cite{ForstSto91} proved local approximation near certain type of CR-singularity. In this paper, we will focus on sets with certain particular forms which has certain CR-singularity. First, we mention a result by Bharali \cite{Bha11}, which was a motivating factor for our study.
	
	\begin{result}[Bharali]\label{R:Bharli}
		Let $\overline{D}$ be a closed disc in $\mathbb{C}$ with center at the origin. Let $F$ be a function defined by $F(z)=\bar{z}^m+R(z), z\in \overline{D}, \text{ where } m\in \mathbb{Z}_{+}$ and  $R$ satisfies 
		
		\begin{align*}
			|R(z)-R(\xi)|< \left|z^{m}-\xi^{m}\right|~~ \forall z,\xi\in \overline{D}:z^m\ne \xi^m.
		\end{align*}
		Then $\Gamma:=Gr_{\overline{D}}(F), \text{the graph of F over }\overline{D}$ is polynomially convex.
		Additionally, we can conclude that $[z,F]_{\overline{D}}=\mathcal{C}(\overline{D})$ in the following cases:
		\begin{itemize}
			\item whenever m=1 (with no further conditions on R)\\
			\item if $m\ge 2, R\in \mathcal{C}^{1}(\overline{D})$ and there exist $\alpha\in (0,1)$ such that $R$ satisfies the stronger estimate:
			$$|R(z)-R(\xi)|< \alpha\left|z^{m}-\xi^{m}\right|~~ \forall  z,\xi\in \overline{D}:z^m\ne \xi^m.$$
		\end{itemize}
	\end{result}
	
	\begin{remark}
		In \Cref{R:Bharli}, the graph has a CR-singularity at the origin.
	\end{remark}	
	\par We now present a generalization of \Cref{R:Bharli} for $n\ge 2.$ This case differs from \Cref{R:Bharli} in the sense that the set of points where the tangent space of the corresponding graph fails to be totally real is non-isolated. It is the graph over an analytic variety of $D$ of co-dimension one.
	\begin{theorem}\label{T:Main Result2}
		\noindent	Let $D$ be an open polydisc in $\mathbb{C}^n$ with centre at the origin and let $\Omega$ be a neighborhood of $\overline{D}$. Let $F_1,F_2,\cdots,F_n$ are functions on $\Omega$ defined by $F_{j}(z)=\bar{z}^{m_j}_{j}+R_{j}(z) \text{ where } m_{j}\in \mathbb{N} \text{ for all } j=1,2,\cdots, n,$ and  $R=(R_1,R_2,\cdots,R_n)$ satisfies 
		\begin{align}\label{E:hypo}
			|R(z)-R(\xi)|\le c \left(\sum_{j=1}^n \left|z_j^{m_{j}}-\xi_j^{m_{j}}\right|^2\right)^{\frac{1}{2}}, 
		\end{align}                                                           \sloppy	$\text{for all } z= (z_1,\cdots,z_n),~~\xi = (\xi_1,\cdots,\xi_n)\in \Omega~~\text{ and for some } c\in(0,1).$ Then $\Gamma:=Gr_{\overline{D}}(F),$ the graph of $F$ over $\overline{D}$ is polynomially convex.
		Additionally, $\text{ if }$ R is $\smoo^1$-smooth on $\Omega$ then we can conclude that
		\begin{align*}      
			[z_{1},\cdots, z_{n}, \bar{z_1}^{m_{1}} + R_{1}(z),\cdots , \bar{z_{n}}^{m_{n}} + R_{n}(z); \overline{D}] = C(\overline{D}).
		\end{align*}
	\end{theorem}
	
	\par We now discuss the situation where $\nu_j\ne 1$ for at least one $j\in\{1,\cdots,n\}.$ For $n=1,$ in \cite{Min76}, Minsker proved the following. 	
	\begin{result}[Minsker]\label{R:Min}
		Let $k,l \in \mathbb{N}$ such that $\gcd(k,l)=1$ and let $\mathbb{D}$ be an open unit disc in $\mathbb{C}.$ Then $[z^k,\bar{z}^l;\overline{\mathbb{D}}] = C(\overline{\mathbb{D}}).$
	\end{result}
	\noindent In \cite{DP1}, De Paepe gave the following generalization of  \Cref{R:Min}:
	\begin{result}\label{R:Dp1}
		Let $F(z)=z^m(1+f(z)),$ $G(z)=\bar{z}^n(1+g(z))$ where $f$ and $g$ are functions defined in a neighborhood  of the origin, of class $C^1,$ with $f(0)=0,$ $g(0)=0.$ If $\gcd(m,n)=1$ and if $D$ is a sufficiently small closed disc around $0$ then $[F,G;D]=C(D).$
	\end{result} 
	
	\noindent Similar type of result can be found in O'Farrell-Preskenis \cite{OP2}, De Paepe \cite{DP1,DP2,DP4,DP5,DP01}, De Paepe and Wiegerinck \cite{DPJW5}, O'Farrell and De Paepe \cite{OFDP3}, Chi \cite{KPC}. But we did not find any result when $\nu_{j}>1,$ for at least one $j\in\{1,2,\cdots,n\},$ and $n>1$ in the literature. In this article, we present the following result.
	In this case, the corresponding graph has non-isolated CR-singularity.	
	\begin{theorem}\label{T:Minsker_Scv}
		\noindent	Let $\mathbb{D}^{n}$ be the open unit polydisc in $\mathbb{C}^n$ and $m_{j}\in \mathbb{N} \text{ for all } j=1,2,\cdots, 2n$ and $\gcd(m_{i},m_{j})=1$ for all $i\ne j.$ Then                                                                          
		\begin{align*}      
			[z^{m_1}_{1},\cdots, z^{m_n}_{n}, \bar{z_1}^{m_{n+1}},\cdots , \bar{z_{n}}^{m_{2n}}; \overline{\mathbb{D}}^{n}] = C(\overline{\mathbb{D}}^{n}).
		\end{align*}
	\end{theorem}
	
	Let $\delta$ be a positive real number. By $D(\delta),$ we denote the open polydisc $D(\delta):=\{z\in \cplx^{n}:|z_{j}|< \delta,j=1,\cdots,n\}$ in $\cplx^{n}$ with polyradius $(\delta,\cdots,\delta).$	
	\begin{theorem}\label{T:Main Result1}
		Let $m_{1}, m_{2},\cdots, m_{2n}$ be positive integers. Fix $\delta_{0}>0.$ Consider a map $R=(R_1,R_2,\cdots,R_n)$ with values in $\mathbb{C}^n,$ defined and of class $C^{1}$ in $D(\delta_{0})$. Suppose that there is a constant $0<c<1$ such that 
		\begin{enumerate}
			\item[(i)]	$|R(z)-R(\xi)|\le c \left(\sum_{j=1}^n \left|z_j^{m_{n+j}}-\xi_j^{m_{n+j}}\right|^2\right)^{\frac{1}{2}}	~~\forall z= (z_1,\cdots,z_n),~~\xi = (\xi_1,\cdots,\xi_n)\in D(\delta_{0});$ 
			\item[(ii)] $R_{j}(z) \sim o(|z_{j}|^{m_{n+j}}) \text{ as } z_j\to 0, \text{ for all } j=1,2,\cdots,n.$ 
		\end{enumerate}
		Then there exist $\delta,$ $0<\delta<\delta_{0},$ such that for 	
		$K:=\{(z^{m_{1}}_{1},\cdots, z^{m_{n}}_{n}, \bar{z_1}^{m_{n+1}} + R_{1}(z),\cdots , \bar{z_{n}}^{m_{2n}} + R_{n}(z)): z\in \overline{D(\delta)}\}, \poly(K)=\smoo(K).$ Furthermore, if $\gcd(m_{i}, m_{j})=1$ $\forall i\ne j,$ then 
		\begin{align*}
			[z^{m_{1}}_{1},\cdots, z^{m_{n}}_{n}, \bar{z_1}^{m_{n+1}} + R_{1}(z),\cdots , \bar{z_{n}}^{m_{2n}} + R_{n}(z); \overline{D(\delta)}] = C(\overline{D(\delta)}).
		\end{align*}
	\end{theorem} 
	\begin{remark}
		Our proof restricts $R_{j}$ to have the property. We do not know what happens when $R_{j}(z) \sim o(|z|^{m_{n+j}})\text{ for all } j=1,2,\cdots,n.$
	\end{remark}	
	\begin{remark}
		If $m_{j}=1~\forall j=1,\cdots,n,$ then, \Cref{T:Main Result2} says that property (ii) is not required, and in this case, we will have a global result.
	\end{remark}

	Let $X$ be a subset of $\cplx^n.$ $X$ is said to be \textit{stratified totally real set} if there exists finitely many closed sets $X_j,j=1,\cdots,N$ such that $X_j\setminus X_{j-1}$ is a totally real set and $X_{0}\subset\cdots\subset X_{N}=X$ with $X=\cup_{j} X_j.$ We now provide few remarks about the proof.
	
	\begin{itemize}
		\item[(i)] For the proof of \Cref{T:Main Result2}, we first show that graph of $F$ is polynomially convex by extending a result  presented in \cite[Proposition 1.7]{Bha11} to higher-dimension. For uniform approximation, we first locate the set of non-totally real points (\Cref{P:Cplx_tngnt}), and then we give a totally real stratification for the graph. Finally, we conclude \Cref{T:Main Result2} by applying \Cref{R:Samuel-Wold} due to Samuelsson and Wold \cite{SaW12}.
		
		\item[(ii)] To prove \Cref{T:Minsker_Scv}, we first use a suitable proper holomorphic map $\Psi$ from $\cplx^{2n}$ to $\cplx^{2n}$ such that	the preimage of 
		$X:=\{(z^{m_{1}}_{1},\cdots, z^{m_{n}}_{n}, \bar{z_1}^{m_{n+1}},\cdots , \bar{z_{n}}^{m_{2n}})\in \cplx^{2n}: z\in \overline{\mathbb{D}}^{n} \}$ is the finite union of $X_{I},$ where each $X_{I}$ is the graph of some linear function. Applying  Stone-Weierstrass approximation theorem, we show that $\poly(X_{I})=\smoo(X_I)$ for each $I.$ After that repeated use of Kallin's lemma (\Cref{R:Sto_apprx.}) gives the polynomial convexity of $\cup_{I}X_{I}.$ Approximation then follows from \Cref{R:Sto_propr}.

		\item[(iii)]  The method of the proof of \Cref{T:Main Result1} is similar to that of \Cref{T:Main Result2}. In this case, we also find a proper holomorphic map $\Phi:\cplx^{2n}\to\cplx^{2n}$ such that the preimage of $X:=\{(z^{m_{1}}_{1},\cdots, z^{m_{n}}_{n}, \bar{z_1}^{m_{n+1}} + R_{1}(z),\cdots , \bar{z_{n}}^{m_{2n}} + R_{n}(z)):z\in \overline{D}\}$ is the finite union of $X_{I},$ where $D$ is some small polydisc centred at the origin. Using \Cref{T:Main Result2}, we show that $\poly(X_{I})=\smoo(X_{I})$ for each $I.$ We find a holomorphic polynomial $p:\cplx^{2n}\to\cplx$ such that each $p(X_{I})$ is contained in an angular sector $\omega_{I}\subset \cplx$ such that $\omega_{I}\cap\omega_{J}=\{0\}~~\forall I\ne J$ and $p^{-1}\{0\}\cap(\cup_{I}X_{I})$ is polynomially convex. Then by repeated application of Kallin's lemma (\Cref{R:Sto_apprx.}), we conclude that $\poly(\cup_{I}X_{I})=\smoo(\cup_{I}X_{I}).$ Since $z^{m_{1}}_{1},\cdots, z^{m_{n}}_{n}, \bar{z_1}^{m_{n+1}} + R_{1}(z),\cdots , \bar{z_{n}}^{m_{2n}} + R_{n}(z)$ separates points $\overline{D}$, we conclude \Cref{T:Main Result1}, by applying \Cref{R:Sto_propr}.
	\end{itemize}
	
	The paper is organized as follows. \Cref{S:technical} collects some earlier results that we will be using in this paper. We also state and proof of a result that characterizes complex points of certain graphs. We also discuss a result that gives a class of continuous functions having polynomially convex graphs. In \Cref{S:Main Result2}, we give a proof of \Cref{T:Main Result2}. \Cref{S:Minsker_SCV} and \Cref{S:Main Result1} are devoted  to the proof of \Cref{T:Minsker_Scv} and \Cref{T:Main Result1} respectively.

	\section{Technical Results}\label{S:technical}

	We begin this section by mentioning few known results that will be used in  the proofs of our theorems. The first one is due to H\"{o}rmander \cite{H}. Let ${\sf psh}(\OM)$ denotes the collection of all plurisubharmonic functions on $\Omega.$
	
	\begin{result}[H\"{o}rmander]\label{R:hormander}
		Let $K$ be a compact subset of a pseudoconvex open set $\OM \subset \cplx^n$. Then $\hull{K}_{\OM} = \hull{K}_{\OM}^P $, where $\hull{K}_{\OM}^P:=\left\lbrace z \in \OM: u(z) \leq\sup\nolimits_{K} u \;\forall u \in {\sf psh}(\OM) \right\rbrace$ and $\hull{K}_\OM:=
		\big\lbrace z\in \OM : |f(z)|\leq \sup\nolimits_{z\in K} |f(z)|\;\forall f\in \hol(\OM)\big\rbrace$.
	\end{result}
	\noindent In case $\OM= \cplx^n,$ \Cref{R:hormander} says that the polynomially convex hull of $K$ is the same as the plurisubharmonically convex hull of $K$.
	\smallskip
	
	We now state a couple of results from Stouts book \cite{Sto07}. The first one \cite[Theorem 1.6.19]{Sto07} is a lemma due to Eva Kallin \cite{Kallin65} and is often referred to as \textit{Kallin’s lemma.} 
	
	\begin{result}\label{R:Sto_apprx.}
		Let $X_1$ and $X_{2}$ be compact, polynomially convex subset of $\cplx^{n}.$ Let $p$ be a polynomial such that $\widehat{p(X_{1})}\cap\widehat{p(X_{2})}\subset\{0\}.$ If the set $p^{-1}(0)\cap(X_1\cup X_2)$ is polynomially convex, then the set $X_{1}\cup X_{2}$ is polynomially convex. If, in addition, $\poly(X_{j})=\smoo(X_{j}), j=1,2,$ then $\poly(X_{1}\cup X_{2})=\smoo(X_{1}\cup X_{2}).$
	\end{result}
	
	\begin{result}\label{R:Sto_propr}
		If $F:\mathbb{C}^n\to \mathbb{C}^n$ is a proper holomorphic map, and if $X\subset \mathbb{C}^n$ is a compact set, then the set $X$ is polynomially convex if and only if the set $F^{-1}(X)$ is
		polynomially convex, and $\mathcal{P}(X) =\mathcal{C}(X)$ if and only if $\mathcal{P}(F^{-1}(X)) =\mathcal{C}(F^{-1}(X)).$
	\end{result}

	\noindent	We now state an approximation theorem on stratified totally real set due to Samuelsson and Wold \cite{SaW12}. Let $\hol(X_{0})$ be the collection of all of all holomorphic function on $X_{0}.$
	\begin{result}\label{R:Samuel-Wold}
		Let $X$ be a polynomially convex compact set in $\mathbb{C}^n$ and assume that there are closed sets $X_{0}\subset\cdots\subset X_{N}=X$ such that $X_{j}\setminus X_{j-1},$ $j=1,\cdots,N,$ is a totally real set. Then $[z_1,\cdots,z_{n}]_X=\{f\in \smoo(X): f|_{X_{0}}\in \hol(X_{0})\}.$ In particular, if $C(X_0)=\mathcal{O}(X_0)$ then $C(X) = [z_1,\cdots, z_n]_X.$
	\end{result}
	
	\smallskip

	\noindent The next lemma that we state and prove gives a characterization when the graph $\Gamma$ in the statement of \Cref{T:Main Result2} is totally real. This will play a vital role in our proof of \Cref{T:Main Result2}. 
	
	\noindent	Let $D$ be an open polydisc in $\mathbb{C}^n$ with center at the origin and let $\Omega$ be a neighborhood of $\overline{D}$. Let $F_1,F_2,\cdots,F_n$ are functions defined by $F_{j}(z)=\bar{z}^{m_j}_{j}+R_{j}(z) \text{ on } \Omega \text{ where } m_{j}\in \mathbb{N} \text{ for all } j=1,2,\cdots, n$ and  $R=(R_1,R_2,\cdots,R_n)$ satisfies (\ref{E:hypo}). Also assume that $R$ is $\smoo^1$-smooth on $\Omega.$ We define a map $\Phi:\Omega\to \mathbb{C}^{2n}$ by
	\begin{align*}
		\Phi(z_1,\cdots,z_n):=(z_1,\cdots,z_n,F_1(z),\cdots , F_{n}(z)).
	\end{align*}
	For $k\le n,$ let $\widetilde{\Omega}$ be an open set in $\cplx^{k}$ and $g:\widetilde{\Omega}\hookrightarrow\Omega$ be an embedding defined by $g(z_1,\cdots,z_k):=(z_1,\cdots,z_k,0,\cdots,0).$ Set $M:=(\Phi\circ g)(\widetilde{\Omega})=Gr_{g(\widetilde{\Omega})}(F),$ which is a real submanifold of $\cplx^{2n}$ of dimension $2k.$  
	
	\begin{proposition}\label{P:Cplx_tngnt}	
		$M$ has complex tangents at $(\Phi\circ g)(z)$ if and only if $z\in  \{(z_1,\cdots,z_k)\in \widetilde{\Omega}:z_1\cdots z_k=0\}.$
	\end{proposition}

	\begin{proof}
		First, we assume that $M$ has complex tangent at some point $z\in \widetilde{\Omega}$. Take $w=\left(\Phi\circ g\right)(z).$
		Let $T_wM$ be the tangent space to $M$ at $w$ viewed as a real-linear subspace of $\cplx^{2n}.$ Since $\Phi\circ g$ is an embedding of $\widetilde{\Omega},$ we have 
		\begin{align*}
			T_wM=\left\{d(\Phi\circ g)|_z(v):v\in \cplx^k\right\}.	
		\end{align*}
		\noindent Since $T_wM$ contains non-trivial complex subspace, then there exists a non-zero vector $\eta\in T_wM$ such that $i\eta$ also belongs to $T_wM.$ We take $\eta=d(\Phi\circ g)|_{z}(v^{0}),$ for some $v^{0}\in\cplx^k\setminus\{0\}$ and $i\eta=d(\Phi\circ g)|_{z}(\omega^{0}),$ for some $\omega^{0}\in\cplx^k\setminus\{0\}.$	
		Therefore, we obtain that
		\begin{align}\label{E:dif_phipi}
			d(\Phi\circ g)|_{z}(\omega^{0})=i d(\Phi\circ g)|_{z}(v^{0}).
		\end{align}
		We denote $z':=(z_1,\cdots,z_k,0,\cdots,0)\in\cplx^n \text{ for }z=(z_1,\cdots,z_k)\in\cplx^k.$
		\begin{align*}
			(\Phi\circ g)(z_1,\cdots,z_k)=(z_1,\cdots,z_k,0,\cdots,0,\bar{z_1}^{m_{1}} +R_{1}(z'),\cdots , \bar{z_{k}}^{m_{k}} + R_{k}(z'),R_{k+1}(z'),\cdots,R_{n}(z')).
		\end{align*}
		
		\noindent The matrix representation of $d(\Phi\circ g)|_{z}$ is
		$$
		\begin{pmatrix}
			\begin{matrix}
				I_{k}  \\ 
			\end{matrix}
			& \vline & 0_{k\times k} \\
			\hline
			\begin{matrix}
				0_{(n-k)\times k} 
			\end{matrix}
			& \vline & 0_{(n-k)\times k} \\
			\hline
			\begin{matrix}
				R_{z}(z')\\
			\end{matrix} & \vline &
			\begin{matrix}
				A(z')+ R_{\bar{z}}(z') \\
			\end{matrix}
		\end{pmatrix}_{2n\times2k},
		$$
		
		\noindent Where\\
		$$	R_{z}(z'):=	
		\begin{pmatrix} 
			\frac{\partial R_{1}}{\partial{z_{1}}}(z') & \frac{\partial R_{1}}{\partial{z_{2}}}(z')&\cdots&  \frac{\partial R_{1}}{\partial{z_{k}}}(z') \\[1.5ex]
			\frac{\partial R_{2}}{\partial{z_{1}}}(z')&  \frac{\partial R_{2}}{\partial{z_{2}}}(z')&\cdots&  \frac{\partial R_{2}}{\partial{z_{k}}}(z')\\[1.5ex]
			
			\vdots &\vdots  &\cdots&\vdots\\
			
			\frac{\partial R_{n}}{\partial{z_{1}}}(z') &  \frac{\partial R_{n}}{\partial{z_{2}}}(z')&\cdots & \frac{\partial R_{n}}{\partial{z_{k}}}(z')& \\\\
			
		\end{pmatrix}_{n\times k},~~ 
		R_{\bar{z}}(z'):=	
		\begin{pmatrix} 
			\frac{\partial R_{1}}{\partial\bar{{z_{1}}}}(z')& \frac{\partial R_{1}}{\partial\bar{z_{2}}}(z')&\cdots&  \frac{\partial R_{1}}{\partial\bar{{z_{k}}}}(z')\\[1.5ex]

			\frac{\partial R_{2}}{\partial\bar{{z_{1}}}}(z')& \frac{\partial R_{2}}{\partial\bar{z_{2}}}(z')&\cdots&  \frac{\partial R_{2}}{\partial\bar{{z_{k}}}}(z')\\[1.5ex]
			
			\vdots &\vdots& \cdots&\vdots   \\
			
			\frac{\partial R_{n}}{\partial\bar{{z_{1}}}}(z')& \frac{\partial R_{n}}{\partial\bar{z_{2}}}(z')&\cdots&  \frac{\partial R_{n}}{\partial\bar{{z_{k}}}}(z')\\\\
			
		\end{pmatrix}_{n\times k},
		$$	
		and
		\begin{align}\label{E:Matrix A}
			A(z')=	
			\begin{pmatrix} 
				m_{1}\bar{z_1}^{m_{1}-1} & 0 & 0 &\cdots & 0 \\[1.5ex]
				0 &	m_{2}\bar{z_2}^{m_{2}-1} & 0 &\cdots & 0\\[1.5ex]
				\vdots & \vdots   & \vdots &   &\vdots\\
				0 & 0 & 0 & \cdots & m_{k}\bar{z_k}^{m_{k}-1}\\[1.5ex]
				0 & 0 & 0 & \cdots & 0\\[1.5ex]
				\vdots & \vdots   & \vdots &   &\vdots\\
				0 & 0 & 0 &\cdots & 0 \\
			\end{pmatrix}_{n\times k},
		\end{align}	
		For any vector $\beta$ in $\mathbb{C}^k,$
		$$d(\Phi\circ g)|_{z}(\beta)=
		\begin{pmatrix}
			\begin{matrix}
				I_{k}  \\ 
			\end{matrix}
			& \vline & 0_{k\times k} \\
			\hline
			\begin{matrix}
				0_{(n-k)\times k} 
			\end{matrix}
			& \vline & 0_{(n-k)\times k} \\
			\hline
			\begin{matrix}
				R_{z}(z')\\
			\end{matrix} & \vline &
			\begin{matrix}
				A(z')+ R_{\bar{z}}(z') \\
			\end{matrix}
		\end{pmatrix}_{2n\times 2k}
		\begin{pmatrix}
			\beta\\\overline{\beta}
		\end{pmatrix}.
		$$
		\noindent This implies
		$$d(\Phi\circ g)|_{z}(\beta)=\left(\beta,0,\cdots,0,R_{z}(z')\beta+R_{\bar{z}}({z'})\bar{\beta}+A(z')\bar{\beta}\right).$$
		\noindent Hence, (\ref{E:dif_phipi}) gives us that
		$$(\omega^{0},0,\cdots,0,R_{z}(z')\omega^{0}+R_{\bar{z}}(z')\bar{\omega^{0}}+A(z')\bar{\omega^{0}})=i(v^{0},0,\cdots,0,R_{z}(z')v^{0}+R_{\bar{z}}(z')\bar{v^{0}}+A({z'})\bar{v^{0}}).$$
		\noindent It follows that $\omega^{0}=iv^{0}$ and 
		$$(R_{z}(z')\omega^{0}+R_{\bar{z}}(z')\bar{\omega^{0}}+A(z')\bar{\omega^{0}})=i(R_{z}(z')v^{0}+R_{\bar{z}}(z')\bar{v^{0}}+A({z'})\bar{v^{0}}).$$
		Putting $\omega^{0}=iv^{0}$ in the above equation, we obtain that

		\begin{align}\label{E:rltn R_A}
			R_{\bar{z}}(z')\bar{v^{0}}+A(z')\bar{v^{0}}=0.
		\end{align}
		By Taylor's formula, for $z=(z_1,\cdots,z_k
		)\in \widetilde{\Omega},~\vartheta=(\vartheta_1,\cdots,\vartheta_{k})\in \mathbb{C}^k, \text{ and } \epsilon$ real,
		
		\begin{align}\label{E:Tylr_frmla}
			&(R\circ g)(z+\eps \vartheta)-(R\circ g)(z)\notag
			=\big((R_1\circ g)(z+\eps \vartheta)-(R\circ g)(z),\cdots,(R\circ g)(z+\eps \vartheta)-(R_n\circ g)(z)\big)\\\notag
			&=\bigg(\sum_{j=1}^{k}\bigg[\delbydel{(R_1\circ g)(z)}{z_j}(\eps\vartheta_j)+\delbydel{(R_1\circ g)(z)}{\bar {z_j}}(\eps\bar{\vartheta_j})\bigg]+o(\eps),\cdots,\sum_{j=1}^{k}\bigg[\delbydel{(R_n\circ g)(z)}{z_j}(\eps\vartheta_j)+\\\notag
			&\delbydel{(R_n\circ g)(z)}{\bar {z_j}}(\eps\bar{\vartheta_j})\bigg]+o(\eps)\bigg)\\\notag
			&=\left(\sum_{j=1}^{k}\bigg[\delbydel{R_1(z')}{z_j}(\eps\vartheta_j)+\delbydel{R_1(z')}{\bar {z_j}}(\eps\bar{\vartheta_j})\bigg]+o(\eps),\cdots,\sum_{j=1}^{k}\bigg[\delbydel{R_n(z')}{z_j}(\eps\vartheta_j)+\delbydel{R_n(z')}{\bar {z_j}}(\eps\bar{\vartheta_j})\bigg]+o(\eps)\right)\\
			&=R_{z}(z')\epsilon\vartheta+R_{\bar{z}}(z')\epsilon\bar{\vartheta}+o(\epsilon),~~\text{ where } z'=(z,0)\in \Omega.\\\notag
		\end{align}
		
		\noindent Applying (\ref{E:hypo}), we have, from (\ref{E:Tylr_frmla}), that		
		\allowdisplaybreaks	
		\begin{align*}
			|R_{z}(z')\epsilon\vartheta+R_{\bar{z}}(z')\epsilon\bar{\vartheta}+o(\epsilon)|
			&\le c \left(\sum_{j=1}^k \left|(z_j+\epsilon\vartheta_j)^{m_{j}}-z_j^{m_{j}}\right|^2\right)^{\frac{1}{2}}\\[1.5ex]
			&\le c \left(\sum_{j=1}^k |\epsilon\vartheta_j|^2 \left(\sum_{l=0}^{m_{j}-1} \left|z_j+\epsilon\vartheta_j\right |^{l} \left|z_j\right|^{m_{j}-l-1}\right)^{2} \right)^{\frac{1}{2}}.
		\end{align*}
		Taking $\epsilon\to 0,$ we get 
		\allowdisplaybreaks
		\begin{align*}
			\left |R_{z}(z')\vartheta+R_{\bar{z}}(z')\bar{\vartheta}\right |&\le c \left(\sum_{j=1}^k |\vartheta_j|^2 \left(\sum_{l=0}^{m_{j}-1} \left|z_{j}\right |^{l} \left|z_j\right|^{m_{j}-l-1}\right)^{2} \right)^{\frac{1}{2}},
		\end{align*}
		i.e.
		\begin{align}\label{E:R+bar(R)}
			\left |R_{z}(z')\vartheta+R_{\bar{z}}(z')\bar{\vartheta}\right |	&\le c \left(\sum_{j=1}^k  \left(m_{j}|\vartheta_j| \left|z_j\right|^{m_{j}-1}\right)^{2} \right)^{\frac{1}{2}} \text{ for all }\vartheta\in \mathbb{C}^k.
		\end{align} 
		
		\noindent Since (\ref{E:R+bar(R)}) is true for all $\vartheta\in\cplx^n,$ we now replace $\vartheta$ by $i\vartheta$ to get
		\begin{align}\label{E:R-bar(R)}
			\left |R_{z}(z')\vartheta-R_{\bar{z}}(z')\bar{\vartheta}\right |	&\le c \left(\sum_{j=1}^k  \left(m_{j}|\vartheta_j| \left|z_j\right|^{m_{j}-1}\right)^{2} \right)^{\frac{1}{2}} \text{ for all }\vartheta\in \mathbb{C}^k.
		\end{align} 
		\noindent In view of (\ref{E:R+bar(R)}) and (\ref{E:R-bar(R)}), we get that
		\begin{align}\label{E:norm_R bar}
			\left |R_{\bar{z}}(z')\bar{\vartheta}\right |	&\le c \left(\sum_{j=1}^k \left(m_{j}|\vartheta_j|\left|z_j\right|^{m_{j}-1}\right)^{2} \right)^{\frac{1}{2}} \text{ for all }\vartheta\in \mathbb{C}^k.
		\end{align}
		
		\noindent Therefore, from (\ref{E:rltn R_A}) and (\ref{E:norm_R bar}), we get that
		\begin{align}\label{E:norm_A}
			\left |A(z')\overline{v^{0}}\right |	&\le c \left(\sum_{j=1}^k  \left(m_{j}|v^{0}_j|\left|z_j\right|^{m_{j}-1}\right)^{2} \right)^{\frac{1}{2}}, \text{ where }  v^{0}=(v^{0}_1,v^{0}_{2},\cdots,v^{0}_k)\in \mathbb{C}^k\setminus\{0\}.
		\end{align}
		\noindent This implies
		\begin{align}\label{E:finl_estmt_nrm A}
			\left(\sum_{j=1}^k  \left(m_{j}|v^{0}_j| \left|z_j\right|^{m_{j}-1}\right)^{2} \right)^{\frac{1}{2}}	&\le c \left(\sum_{j=1}^k \left(m_{j}|v^{0}_j| \left|z_j\right|^{m_{j}-1}\right)^{2} \right)^{\frac{1}{2}}.
		\end{align}

		Now, if $\left(\sum_{j=1}^k  \left(m_{j}|v^{0}_j| \left|z_j\right|^{m_{j}-1}\right)^{2} \right)^{\frac{1}{2}}\ne 0,$ then we obtain that $c\geq 1,$ which is a contradiction to our assumption. Hence, $\left(\sum_{j=1}^k  \left(m_{j}|v^{0}_j| \left|z_j\right|^{m_{j}-1}\right)^{2} \right)^{\frac{1}{2}}= 0.$ Again, viewing $A(z')$ as a $\cplx$-linear map from $\cplx^k$ to $\cplx^n$ with $v^{0}\in\ker A(z'),$ from (\ref{E:Matrix A}) we get that $\rank A(z')<k$ and hence $z_1 z_2\cdots z_{k}=0.$\\
		\smallskip
		\par Conversely, assume that  $p\in  \left\{z=(z_1,z_2,\cdots,z_k)\in \widetilde{\Omega}:z_1 z_2 \cdots z_k=0\right\}.$ We need to show that $M$ has complex tangent at $(\Phi\circ g)(p).$ Without loss of generality, we take $p=(p_1,p_2,\cdots,p_{k-1},0).$ Suppose $M$ does not have complex tangent at $(\Phi\circ g)(p)$. Thus, we have $(A(p')+ R_{\bar{z}}(p'))v\ne 0~~\forall v\in \cplx^k\setminus\{0\},$ where $p'=(p,0)\in \Omega.$ Choose $v=(0,0,\cdots,0,1)\in\cplx^k$. Since $A(p')v=0,$ we have $R_{\bar{z}}(p')v\ne 0.$ Using (\ref{E:norm_R bar}), we get that $|R_{\bar{z}}(p')v|= 0.$ This is a contradiction. Therefore, $M$ has a complex tangent at $(\Phi\circ g)(p).$
	\end{proof}
	\smallskip
	
	Before going to the next section, we state and prove a lemma that will be useful in the proof of \Cref{T:Main Result1}.
	
	\begin{lemma}\label{L:Argumnt}	
		\sloppy Assume that $m_{j}\in \mathbb{N} \text{ and }\alpha_{j}\in \mathbb{N}\cup \{0\},j=1,\cdots ,2n,$ with $\gcd\left(m_i,m_j\right)=1$ for ${i\ne j},$ and $\gcd(m_i,\alpha_{j})=1 ~~\forall\alpha_{j}\ne 0.$ Let $(t_{1},t_{2},\cdots,t_{2n}) ,(t'_{1},t'_{2},\cdots,t'_{2n})\in \left\{0,1,\cdots,m_1-1\right\}\times \cdots\times \left\{0,1,\cdots,m_{2n}-1\right\}.$ Assume that $\{i_1,\cdots i_k\}\subset \{1,\cdots,2n\}$ and if there exist $j_{0}\in \{1,\cdots,2n\}\setminus\{i_1,\cdots i_k\}$ such that $t_{j_{0}}\ne t'_{j_0},$ and $\alpha_{j_{0}}\ne 0,$ then 	
		\begin{align*}
			{{\sum^{2n}_{\substack{j=1\\
							j\notin\{i_1,\cdots,i_{k}\}	\\
					}}
					\frac{t_j\alpha_{j}}{m_j}}}\ne 
			{{\sum^{2n}_{\substack{j=1\\
							j\notin\{i_1,\cdots,i_{k}\}	\\
					}}
					\frac{t'_j\alpha_{j}}{m_j}}}.
		\end{align*}
	\end{lemma}	
	
	\begin{proof}
		\noindent Assume that there exist $j_{0}\in \{1,\cdots,2n\}\setminus\{i_{1},\cdots,i_{k}\}$ such that $t_{j_0}\not=t'_{j_0},$ $\alpha_{j_{0}}\not =0$ and
		
		\begin{align*}
			{{\sum^{2n}_{\substack{j=1\\
							j\notin\{i_1,\cdots,i_{k}\}	\\
					}}
					\frac{t_j\alpha_{j}}{m_j}}}=
			{{\sum^{2n}_{\substack{j=1\\
							j\notin\{i_1,\cdots,i_{k}\}	\\
					}}
					\frac{t'_j\alpha_{j}}{m_j}}}.
		\end{align*}
		\noindent Therefore,
		\allowdisplaybreaks 
		\begin{align*}
			&\frac{\alpha_{j_{0}}}{m_{j_{0}}}\left( t_{j_0}'-t_{j_0}\right)=	{{\displaystyle\sum^{2n}_{\substack{j=1\\
							j\notin\{i_1,\cdots,i_{k},j_{0}\}	\\
					}}
					\frac{\alpha_{j}}{m_j}\bigg(t_j-t'_{j}}}\bigg)\\	
			&=\frac{{{\displaystyle\sum_{\substack{
								j\notin\{i_1,\cdots,i_{k},j_{0}\}	\\
						}}
						\bigg(\alpha_{j}m_1m_2\cdots\widehat{m_j}\cdots m_{n}(t_j-t'_j)}}\bigg)} {\displaystyle	\prod_{\substack{
						j\notin\{i_1,\cdots,i_{k},j_{0}\}	\\
				}}
				m_j},
		\end{align*}
		where $\widehat{m_j}$ denote the absence of $j^{th}$ term.			
		\noindent This implies 
		\begin{align*}
			m_{j_{0}}	\bigg({{\sum_{\substack{
							j\notin\{i_1,\cdots,i_{k},j_{0}\}	\\
					}}
					\alpha_{j}m_1m_2\cdots\widehat{m_j}\cdots m_{n}(t_j-t'_j)}}\bigg)=\alpha_{j_{0}}(t'_{j_0}-t_{j_{0}})	\prod_{\substack{
					j\notin\{i_1,\cdots,i_{k},j_{0}\}	\\
			}}
			m_j,
		\end{align*}
		\noindent From above, we can say that $m_{j_{0}} $ divides $ \displaystyle\alpha_{j_{0}}(t'_{j_0}-t_{j_{0}}) \prod_{\substack{
				j\notin\{i_1,\cdots,i_{k},j_{0}\}\\
		}}m_j.$ Since $\gcd(m_{i},m_{j})=1$ for $i\ne j$ and $\gcd(m_i,\alpha_{j})=1,$ we get that
		$m_{j_{0}} $ divides $(t'_{j_0}-t_{j_{0}}).$ This is not possible because $0\le t_{j_0},t'_{j_{0}}<m_{j_{0}}.$
		
	\end{proof}

	\par In \cite{Bha11}, Bharali introduced a technique to study the polynomial convexity of graphs of functions in one variable. We observe here that the same technique can be generalized to higher-dimension. We will now state and prove this result. For that, we need some notations. 
	\begin{itemize}
		\item For a compact set $K\subset\cplx^n,$ by $\hol(K;\cplx^{N})$ we define the set of holomorphic functions from some neighborhood of $K$ to $\cplx^N.$ If $N=1,$ we simply denote it $\hol(K);$ And for fix $\xi\in K,$  we define a sub-class of $\hol(K;\cplx^{N})$ by $\hol_{\xi}(K;\cplx^{N}):=\left\{F\in\hol(K;\cplx^{N}):F(\xi)=0\right\}.$
		\item For $A\in\hol(K) \text{ and } F=(f_1,\cdots,f_N)\in \hol(K;\cplx^{N})$ by $A(z)F(z)$, we mean $\left(A(z)f_1(z),\cdots,\right.$ $\left. A(z)f_{N}(z)\right)$ and by $P(z)Q(z),$ we mean $\sum_{j=1}^{N}p_j(z)q_j(z)$ for $P=(p_1,\cdots,p_N)$ and $Q=(q_1,\cdots
		,q_N);$
		\item Let $K$ be a subset of $\cplx^n$ and $F:=(f_1,f_2,\cdots,f_N):K\to\cplx^N$ be a function. We will denote the graph of $F$ over $K$ by $Gr_{K}F$ or by $Gr_{K}(f_1,f_2,\cdots,f_N).$
	\end{itemize}

	\noindent \begin{proposition}\label{P:poly cnvx1}
		Let $N\ge n\ge 1$ and $F\in\smoo(\overline{\Omega};\cplx^{N}),$ where $\Omega$ is a bounded domain in $\cplx^n$ with $\overline{\Omega}$ is polynomially convex, and let $\xi\in\overline{\Omega}.$ Suppose there exist a constant $p\ge 2,$ a nowhere vanishing function $A\in \hol(\overline{\Omega}),$ and mappings $G,H\in\hol_{\xi}(\overline{\Omega};\cplx^{N})$ such that 
		\begin{align}\label{Eq:Poly cnvx1}
			\left|A(z)F(z)-A(\xi)F(\xi)+G(z)\right|^p\le \rl\left(H(z)\left(A(z)F(z)-A(\xi)F(\xi)+G(z)\right)\right)~~\forall z\in \overline{\Omega}.
		\end{align}
		Then $Gr_{\overline{\Omega}}(F)$ is polynomially convex. 
	\end{proposition}
	
	\begin{proof}
		Since $A\in \hol(\overline{\Omega})$ and $G,H\in\hol_{\xi}(\overline{\Omega};\cplx^{N}),$ there exists a neighborhood $U$ of $\overline{\Omega}$ such that the functions $A,G,H$ are holomorphic on $U.$ Since $\overline{\Omega}$ is polynomially convex, there exist an open polynomial polyhedron $\triangle$ such that $\overline{\Omega}\subset \triangle\subset U.$ Define a map $\Psi:\triangle\times \mathbb{C}^N\to\mathbb{R}$ by 
		\begin{align*}
			\Psi(z,w):=\left|A(z)w-A(\xi)F(\xi)+G(z)\right|^p
			-\rl\left(H(z)\left(A(z)w-A(\xi)F(\xi)+G(z)\right)\right).
		\end{align*}
		\noindent Since $\rl\left(H(z)\left(A(z)w-A(\xi)F(\xi)+G(z)\right)\right)$ is pluriharmonic in $z$ and $w$ on $\triangle\times\mathbb{C}^N,$ $\Psi$ is plurisubharmonic on $\triangle\times\mathbb{C}^N.$ We set $\Gamma:=Gr_{\overline{\Omega}}(F)$. Since $\overline{\Omega}$ is polynomially convex, we can say that $\widehat{\Gamma}\Subset {\overline{\Omega}}\times\mathbb{C}^N.$ 
		\noindent We consider the set
		\begin{align*}
			\mathcal{S}:=\big\{(z,w)\in {\overline{\Omega}}\times \mathbb{C}^{N} & :\Psi(z,w)\le 0\big\}.
		\end{align*}
		We now claim that $\widehat{\Gamma}\subset\mathcal{S}.$ To prove this claim assume $(z_0,w_0)\in \widehat{\Gamma}$ but $(z_0,w_0)\notin \mathcal{S}.$ Therefore, $	\Psi(z_0,w_0)>0$ and by assumption (\ref{Eq:Poly cnvx1}), we get that $\sup_{\Gamma}\Psi(z,w)\le 0.$ Since $\triangle$ is an open Runge and Stein neighborhood of $\overline{\Omega}$ with $\Psi\in {\sf psh}(\triangle\times \cplx^{N})$ such that $\Psi(z_0,w_0)>\sup_{\Gamma}\Psi(z,w),$ we get that $(z_0,w_0)\notin \widehat{\Gamma}_{\mathcal{O}(\triangle\times \cplx^{N})}$ and so $(z_0,w_0)\notin \widehat{\Gamma}.$ Therefore, we obtain that $\widehat{\Gamma}\subset\mathcal{S}.$ This also implies that  
		\begin{align*}
			\left|A(z)w-A(\xi)F(\xi)+G(z)\right|^p
			&\le|H(z)|\left|A(z)w-A(\xi)F(\xi)+G(z)\right|~~\forall (z,w)\in\widehat{\Gamma}.
		\end{align*} 
		
		\noindent Since $p\ge 2$, we obtain that
		\begin{align}\label{Set:gamma_hat}
			\widehat{\Gamma}\subset\left\{(z,w)\in \overline{\Omega}\times \mathbb{C}^{N}:\left|A(z)w-A(\xi)F(\xi)+G(z)\right|
			\le |H(z)|^{\frac{1}{p-1}}\right\}.
		\end{align}
		\par Our next claim is: $\widehat{\Gamma}\cap\left(\{\xi\}\times\mathbb{C}^N\right)=\{(\xi,F(\xi))\}.$
		To prove this claim, assume $(\alpha,\beta)\in \widehat{\Gamma}\cap\left(\{\xi\}\times\mathbb{C}^N\right).$ This implies $\alpha=\xi.$ Since $(\xi,\beta)\in \widehat{\Gamma},$ we get from (\ref{Set:gamma_hat}) that
		\begin{align*}	
			\left|A(\xi)\beta-A(\xi)F(\xi)+G(\xi)\right|
			\le |H(\xi)|^{\frac{1}{p-1}}.	
		\end{align*}
		\noindent In view of the assumption $H(\xi)=0=G(\xi),$ we obtain that
		$|A(\xi)||\beta-F(\xi)|=0.$ Since $A$ is nowhere vanishing, we obtain $\beta=F(\xi).$
		This proves the claim. Since, by assumption, for each $\xi\in \overline{\Omega},$ such functions $A,G,H$ exists, therefore, we obtain that	$\widehat{\Gamma}=\widehat{\Gamma}\cap\left(\cup_{\xi\in {\overline{\Omega}}}\left(\{\xi\}\times\mathbb{C}^N\right)\right)=\cup_{\xi\in {\overline{\Omega}}}\{(\xi,F(\xi))\}=\Gamma.$
		
	\end{proof}

	\section{The proof of Theorem~\ref{T:Main Result2}}\label{S:Main Result2}
	
	First, we state and prove a lemma that is crucial to our proof.		
	\begin{lemma}\label{L:main lemma}
		Let $\Omega$ be a bounded domain in $\cplx^n$ such that $\widehat{\overline{\Omega}}=\overline{\Omega}$ and let $F_1,F_2,\cdots,F_n$ be functions on $\overline{\Omega}$ defined by $F_{j}(z)=\bar{z}^{m_j}_{j}+R_{j}(z),~~j=1,\cdots,n,$ and  $m_{j}\in \mathbb{N}.$ Assume that there exists a real number $c\in(0,1)$ such that the map $R=(R_1,R_2,\cdots,R_n)$ satisfies 
		\noindent \begin{align*}
			|R(z)-R(\xi)|\le c \left(\sum_{j=1}^n \left|z_j^{m_{j}}-\xi_j^{m_{j}}\right|^2\right)^{\frac{1}{2}}~~ \forall z,\xi \in \overline{\Omega}.
		\end{align*}
		Then $Gr_{\overline{\Omega}}(F_1,F_2,\cdots,F_n)$ is polynomially convex.
	\end{lemma}

	\begin{proof}
		Fixing $\xi\in \overline{\Omega}$ we define $\Psi_{\xi}:\overline{\Omega}\to \cplx$ by
		\allowdisplaybreaks
		\begin{align*}
			\Psi_{\xi}(z)&:=\sum_{j=1}^{n}\left(z^{m_j}_j-\xi^{m_j}_{j}\right)\left( F_j(z)-F_j(\xi)\right).
		\end{align*}
		\noindent Putting $F_{j}(z)=\bar{z}^{m_j}_{j}+R_{j}(z),$ we get that	
		\begin{align}\label{E:fn_si}
			\Psi_{\xi}(z)=\sum_{j=1}^{n}\left |z^{m_{j}}_j-\xi_{j}^{m_{j}}\right |^2+\sum_{j=1}^{n}\left(z^{m_j}_j-\xi^{m_j}_{j}\right)\left( R_j(z)-R_j(\xi)\right)~~\forall z\in \overline{\Omega}.
		\end{align}
		\noindent Therefore, the real part of  $\Psi_{\xi}(z)$ is
		\begin{align*}
			\rl(\Psi_{\xi}(z))=\sum_{j=1}^{n}\left |z^{m_{j}}_j-\xi_{j}^{m_{j}}\right |^2+\rl\left(\sum_{j=1}^{n}\left(z^{m_j}_j-\xi^{m_j}_{j}\right)\left( R_j(z)-R_j(\xi)\right)\right).
		\end{align*}	
		\noindent	We now compute:
		\begin{align*}
			\rl\left(\sum_{j=1}^{n}\left(z^{m_j}_j-\xi^{m_j}_{j}\right)\left( R_j(z)-R_j(\xi)\right)\right)
			&\leq \sum_{j=1}^{n}\left|\left(z^{m_j}_j-\xi^{m_j}_{j}\right)\left( R_j(z)-R_j(\xi)\right)\right|\\
			&\leq \left(\sum_{j=1}^{n}\left|z^{m_j}_j-\xi^{m_j}_{j}\right|^2\right)^{\frac{1}{2}}|R(z)-R(\xi)|\\
			&\leq c \left(\sum_{j=1}^{n}\left|z^{m_j}_j-\xi^{m_j}_{j}\right|^2\right).
		\end{align*}
		\noindent Here we use Cauchy-Schwarz for the second inequality and the last inequality follows from (\ref{E:hypo}). From above calculations, we have
		
		\begin{align*}
			- c \left(\sum_{j=1}^{n}\left|z^{m_j}_j-\xi^{m_j}_{j}\right|^2\right)	\leq \rl\left(\sum_{j=1}^{n}\left(z^{m_j}_j-\xi^{m_j}_{j}\right)\left( R_j(z)-R_j(\xi)\right)\right)&\leq c \left(\sum_{j=1}^{n}\left|z^{m_j}_j-\xi^{m_j}_{j}\right|^2\right).
		\end{align*}
		
		Therefore, we obtain that 
		\begin{align}\label{E:rl_prt_si}
			0\leq (1-c)\left(\sum_{j=1}^{n}\left|z^{m_j}_j-\xi^{m_j}_{j}\right|^2\right)	\leq \rl(\Psi_{\xi}(z))\leq (1+c)\left(\sum_{j=1}^{n}\left|z^{m_j}_j-\xi^{m_j}_{j}\right|^2\right).
		\end{align}
		
		Let us compute:	
		\allowdisplaybreaks
		\begin{align*}
			\left|F(z)-F(\xi)\right|&=\left( \sum_{j=1}^{n}\left|\bar{z_j}^{m_j}+R_j(z)-\bar{\xi_j}^{m_j}-R_j(\xi)\right|^2\right)^{\frac{1}{2}}\\
			&\le\left(\sum_{j=1}^{n}\left(\left|\bar{z_j}^{m_j}-\bar{\xi_j}^{m_j}\right|+\left|R_j(z)-R_j(\xi)\right|\right)^2\right)^{\frac{1}{2}}\\
			&\le\left(\sum_{j=1}^{n}\left|\bar{z_j}^{m_j}-\bar{\xi_j}^{m_j}\right|^2\right)^{\frac{1}{2}}+\left(\sum_{j=1}^{n}\left|R_j(z)-R_j(\xi)\right|^2\right)^{\frac{1}{2}}.
		\end{align*}
		Here we use Minkowski inequality for the last inequality. Using (\ref{E:hypo}), we obtain from above that
		\allowdisplaybreaks
		\begin{align}\label{E:prefnl_estmt_F}
			\left|F(z)-F(\xi)\right|^2\leq (1+c)^2\sum_{j=1}^{n}\left|{z_j}^{m_j}-{\xi_j}^{m_j}\right|^2~~ \forall (z,\xi)\in \overline{\Omega}\times\overline{\Omega}.
		\end{align}
		
		\noindent In view of (\ref{E:rl_prt_si}) and (\ref{E:prefnl_estmt_F}), we get that
		\begin{align}\label{E:final estmate of F}
			\left|F(z)-F(\xi)\right|^2
			&\le \frac{(1+c)^2}{(1-c)}\rl(\Psi_{\xi}(z))~~ \forall (z,\xi)\in \overline{\Omega}\times\overline{\Omega}.
		\end{align}	
		\noindent	Therefore, we obtain that $	|F(z)-F(\xi)|^2
		\le \text{C}\rl\left(\Psi_{\xi}(z)\right),$  where  C is a constant on the right side of (\ref{E:final estmate of F}) which is independent of $z.$ Hence, by \Cref{P:poly cnvx1}, we conclude that $Gr_{\overline{\Omega}}(F)$ is polynomially convex.
	\end{proof}		
	
	\begin{proof}[Proof of Theorem~\ref{T:Main Result2}]
		In view of \Cref{L:main lemma}, we know that  $\Gamma:=Gr_{\overline{D}}(F),$ the graph of $F$ over $\overline{D}$ is polynomially convex. 
		\par We now assume that $R$ is $\smoo^1$-smooth on $\Omega.$ We wish to use \Cref{R:Samuel-Wold} for approximation. For that we require a suitable totally real stratification of $\Gamma.$
		\smallskip 
		For each subset $\{i_1,\cdots,i_k\}\subset\{1,\cdots,n\},$ we define
		
		\begin{align*}
			\sigma_{(i_1,\cdots,i_k)}(z_1,\cdots,z_{n}):=	\prod^{n}_{\substack{j=1\\
					j\notin\{i_1,\cdots,i_{k}\}	\\
			}}
			z_{j}.
		\end{align*}
		\noindent We now give a stratification of $\Omega$ as follows: 	
		\begin{align*}
			Z_{n}&:=\Omega;\\
			Z_{n-1}&:=\left\{z\in\Omega:	\prod^{n}_{\substack{j=1\\
			}}
			z_{j}=0\right\};\\
			Z_{n-2}&:=\left\{z\in\Omega: 	\sigma_{(i_1)}=0,z_{i_1}=0, 1\le i_1\le n	\right\};\\
			Z_{n-(k+1)}&:=\left\{z\in\Omega: 	\sigma_{(i_1,\cdots,i_{k})}=0,z_{i_l}=0,l=1,\cdots,k,~~ 1\le i_1\ne \cdots\ne i_{k}\le n\right\},1\le k\le n-1.
		\end{align*}
		
		\noindent Clearly,	\begin{align*}
			Z_{0}&:=\left\{z\in\Omega: 	\sigma_{(i_1,i_2,\cdots,i_{n-1})}=0, 1\le i_1\ne i_2\ne\cdots\ne i_{n-1}\le n\right\}=\{(0,\cdots,0)\}.
		\end{align*}
		We define our stratification of $Gr_{\overline{D}}(F)$ as follows:
		\begin{align*}
			X_0&:= Gr_{Z_0\cap\overline{D}}(F);\\
			X_1&:= Gr_{Z_1\cap\overline{D}}(F);\\
			\vdots\\
			X_n&:= Gr_{Z_n\cap\overline{D}}(F)=Gr_{\overline{D}}(F).
		\end{align*}
		
		We claim that $X_{k} \setminus X_{k-1}$ is totally real for each $k=1,\cdots,n.$ First we take $k=n$ and $(p,F(p))\in X_{n}\setminus X_{n-1}.$ This implies $p=(p_1,\cdots,p_{n})\in Z_{n}=\Omega.$
		By taking $\widetilde{\Omega}=\Omega$ and $g=Id$ in \Cref{P:Cplx_tngnt}, we get that $(\Phi\circ g)(\widetilde{\Omega})=Gr_{\Omega}(F)$ and $Gr_{\Omega}(F)$ is totally real at $(p,F(p))$ if and only if $p_{1}\cdots p_{n}\ne 0$ i.e $p\notin Z_{n-1}.$ Therefore, $X_{n}\setminus X_{n-1}$ is totally real. Similarly, fix $k\in \{1,\cdots,n\}$ and we take $(p,F(p))\in X_{k}\setminus X_{k-1}.$ This implies $p\in Z_{k}.$ Then there exists a set $\{i_1,\cdots,i_{k}\}\subset \{1,\cdots,n\}$ such that $p_{j}=0~~ \forall {j}\in \{i_1,\cdots,i_{k}\}.$ Without loss of generality, we can assume that that $i_{j}=n-k+j.$ Therefore, $p=(p_1,\cdots,p_{n-k},0,\cdots,0).$ Since \Cref{P:Cplx_tngnt} is true for any $k\le n,$ $Gr_{Z_{k}}(F)$ is totally real if and only if $p_1\cdots p_{n-k}\not =0$ i.e $p\notin Z_{k-1}.$ Therefore, $X_{k} \setminus X_{k-1}$ is totally real.
		
		\noindent So far, we have the following:		
		\begin{itemize}
			\item $X_{n}$ is polynomially convex;
			\item $X_{0}\subset\cdots\subset X_{n}=Gr_{\overline{D}}(F)$ with $X=\cup_{j} X_j;$
			\item $X_{j} \setminus X_{j-1}$ is totally real $\forall {j}\in \{1,\cdots,n\};$
			and 
			\item $X_{0}=\{0\}.$
		\end{itemize}		
		We now apply \Cref{R:Samuel-Wold} to conclude that 
		\begin{displaymath}
			[z_{1},\cdots, z_{n}, \bar{z_1}^{m_{1}} + R_{1}(z),\cdots , \bar{z_{n}}^{m_{n}} + R_{n}(z); \overline{D}] = C(\overline{D}).
		\end{displaymath}
	\end{proof}

	\section{Proof of \Cref{T:Minsker_Scv}}\label{S:Minsker_SCV}
	Before going into the proof of \Cref{T:Minsker_Scv}, we need to do some preparation and we need to prove some lemmas. 
	\par Let
	\begin{align*}
		X:=\{(z^{m_{1}}_{1},\cdots, z^{m_{n}}_{n}, \bar{z_1}^{m_{n+1}},\cdots , \bar{z_{n}}^{m_{2n}})\in \cplx^{2n}: z\in \overline{\mathbb{D}}^{n} \},
	\end{align*} and we define $\Phi:\mathbb{C}^{2n}\to \mathbb{C}^{2n}$ by 
	\begin{align*}
		\Phi\left(z_1,\cdots, z_{n},w_{1},\cdots,w_{n}\right)=\left(z^{m_{1}}_1,\cdots, z^{m_{n}}_{n},w^{m_{n+1}}_{1},\cdots,w^{m_{2n}}_{n}\right).
	\end{align*}
	\noindent Therefore,
	$\Phi^{-1}(X)=\cup_{I}X_{I}, \text{ where }$ $I\in \{0,1,\cdots,m_1-1\}\times \cdots\times \{0,1,\cdots,m_{2n}-1\}$ and for $I=(t_{1},t_{2},\cdots ,t_{2n}),$ 
	\allowdisplaybreaks
	\begin{align*}
		X_{I}:=\left\{\left(e^{\frac{2i\pi t_{1}}{m_1} }z_{1},\cdots, e^{\frac{2i\pi t_{n}}{m_n}}z_{n}, e^{\frac{2i\pi t_{n+1}}{m_{n+1}} }\bar{z_1},\cdots , e^{\frac{2i\pi t_{2n}}{m_{2n}} }\bar{z_{n}}\right): z\in \overline{\mathbb{D}}^{n} \right\}.
	\end{align*}
	We will denote $X_{(0,\cdots,0)}$ by $X_{0}.$ For each subset $\{i_1,\cdots,i_k\}\subseteq\{1,\cdots,2n\}$ and $X_{I}$ as above, we denote
	\begin{align*}
		X^{(i_1,\cdots,i_k)}_{I}:=\bigg\{\bigg(e^{\frac{2i\pi t_{1}}{m_1} }z_{1},\cdots, e^{\frac{2i\pi t_{n}}{m_n}}z_{n}, e^{\frac{2i\pi t_{n+1}}{m_{n+1}} }\bar{z_1},&,\cdots ,  e^{\frac{2i\pi t_{2n}}{m_{2n}} }\bar{z_{n}}\bigg):\\ &z\in \overline{\mathbb{D}}^{n}, z_{i_l}=0,~l=1,\cdots,k \bigg\}.
	\end{align*}	
	We consider the polynomial
	\begin{align*}
		p_{k}(z,w))=	\prod^{n}_{\substack{j=1\\
				j\notin\{i_1,\cdots,i_{k}\}	\\
		}}
		z_{j}w_{j}.
	\end{align*}
	We now prove three lemmas that are crucial in the proof of \Cref{T:Minsker_Scv}.
	
	\begin{lemma}\label{L:KLp1M}
		Let $p_{k}$ and $X^{(i_1,\cdots,i_k)}_{I}$ be as above. Then
		\begin{itemize}
			\item[(i)]$p_{k}(X^{(i_1,\cdots,i_k)}_{I})\subset L^{(i_1,\cdots,i_k)}_{I}$, where $L^{(i_1,\cdots,i_k)}_{I}$ is the half line through the origin with argument $\left({{\displaystyle2\pi\sum^{2n}_{\substack{j=1\\
							j\notin\{i_1,\cdots,i_{k}\}	\\
					}}
					\frac{t_j}{m_j}}}\right);$
			\item[(ii)]	
			$L^{(i_1,\cdots,i_k)}_{I}\cap L^{(i_1,\cdots,i_k)}_{J}=\{0\}~~\text{ for distinct } X^{(i_1,\cdots,i_k)}_{I} \text{ and } X^{(i_1,\cdots,i_k)}_{J};\text{ we also have }$
			
			\item[(iii)]
			$p_{k}^{-1}\{0\}\cap X^{(i_1,\cdots,i_k)}_{I}=\bigg\{\bigg(e^{\frac{2i\pi t_{1}}{m_1} }z_{1},\cdots, e^{\frac{2i\pi t_{n}}{m_n}}z_{n}, e^{\frac{2i\pi t_{n+1}}{m_{n+1}} } \bar{z_1},\cdots ,e^{\frac{2i\pi t_{2n}}{m_{2n}} } \bar{z_{n}}\bigg):\\ z\in \overline{\mathbb{D}}^{n}, z_{i_1}=0,\cdots,z_{i_k}=0, \displaystyle\prod^{n}_{\substack{j=1\\
					j\notin\{i_1,\cdots,i_{k}\}	\\
			}}
			z_{j} =0\bigg\}.$
		\end{itemize}
	\end{lemma}	
	
	\begin{proof}

		\begin{itemize}
			\item[(i)] Let $(z,w)\in X^{(i_1,\cdots,i_k)}_{0}.$
			\noindent Then
			\begin{align}\label{E:Img_X01_pt}
				p_{k}(z,w)&=\prod^{n}_{\substack{j=1\\
						j\notin\{i_1,\cdots,i_{k}\}	\\
				}}
				|z_{j}|^2
			\end{align}
			
			Writing $p_{k}(z,w)=u+iv,$ we get from (\ref{E:Img_X01_pt}) that $v=0 \text{ and } u\ge 0.$ Therefore,
			\begin{align}\label{E:Img_X01_Set}
				p_{k}\left(X^{(i_1,\cdots,i_k)}_{0}\right)\subset\{u+iv\in \cplx: v=0, u\ge 0\}.
			\end{align}
			
			\noindent Let $I=(t_{1},t_{2},\cdots ,t_{2n})$ be an arbitrary element of $ \{0,1,\cdots,m_1-1\}\times \{0,1,\cdots,m_2-1\}\times \cdots\times \{0,1,\cdots,m_{2n}-1\}.$
			For $(z,w)\in X_I,$ we get that
			\begin{align}\label{E:Zeros_pk_X_I}
				p_k(z,w)=e^{i\alpha}
				\prod^{n}_{\substack{j=1\\
						j\notin\{i_1,\cdots,i_{k}\}	\\
				}}
				|z_{j}|^{2}, \text{ where } \alpha={\displaystyle\left({2\pi	\sum^{2n}_{\substack{j=1\\
								j\notin\{i_1,\cdots,i_{k}\}	\\
						}}
						\frac{t_j}{m_j}}\right)}
			\end{align}
			\noindent Equality (\ref{E:Zeros_pk_X_I}) says that $p_{k}\left(X^{(i_1,\cdots,i_k)}_{I}\right)$ lies on the half real line passing through the origin with argument
			$\left({2\displaystyle\pi\sum^{2n}_{\substack{j=1\\
						j\notin\{i_1,\cdots,i_{k}\}	\\
				}}
				\frac{t_j}{m_j}}\right).$
			
			\item[(ii)]For each $I,$ we denote each of the above line by $L^{\{i_1,\cdots,i_{k}\}}_{I}.$ For $\{i_1,\cdots,i_{k}\}=\emptyset,$ we denote these half line $L^{\{i_1,\cdots,i_{k}\}}_{I}$ simply by $L_{I}.$ By putting $\alpha_{j}=1~~~ \forall {j}\in \{1,\cdots,2n\}\setminus\{i_1,\cdots,i_{k}\}$ in \Cref{L:Argumnt}, we get that $L^{\{i_1,\cdots,i_{k}\}}_{I}\cap L^{\{i_1,\cdots,i_{k}\}}_{J}=\{0\}.$
			\item[(iii)] Proof follows from (\ref{E:Zeros_pk_X_I}).
		\end{itemize}
	\end{proof}
	
	\begin{lemma}\label{L:Mk-th term}
		Assume that $\{l_1,\cdots,l_{k}\}\subset\{1,\cdots,n\}$ and $\cup_{I}X_{I}^{(l_1,\cdots,l_{k},j)}$ is polynomially convex for each $j\in\{1,\cdots,n\}\setminus\{l_1,\cdots,l_{k}\},$. Then $\cup_{j\in \{1,\cdots,n\}\setminus\{l_1,\cdots,l_{k}\}}	\left(\cup_{I}X_{I}^{(l_1,\cdots,l_{k},j)}\right)$ is polynomially convex.	
	\end{lemma}
	
	\begin{proof}
		Without loss of generality, we assume that $1\le<l_2<\cdots<l_k\le n.$ Clearly,
		\begin{align*}
			\cup_{j\in \{1,\cdots,n\}\setminus\{l_1,\cdots,l_{k}\}}	\bigg(\cup_{I}X_{I}^{(l_1,\cdots,l_{k},j)}\bigg)=\cup_{j=1}^{k+1}A_{l_j},
		\end{align*}
		where 
		\begin{align*}
			A_{l_{1}}:&=\cup_{j=1}^{l_1-1}	\bigg(\cup_{I}X_{I}^{(l_1,\cdots,l_{k},j)}\bigg);\\
			A_{l_{r+1}}:&=	\cup_{j=l_{r}+1}^{l_{r+1}-1}	\bigg(\cup_{I}X_{I}^{(l_1,\cdots,l_{k},j)}\bigg)~\text{ for }r=1,\cdots,k-1;\text{ and }\\
			A_{l_{k+1}}:&=	\cup_{j=l_{k+1}}^{n}	\bigg(\cup_{I}X_{I}^{(l_1,\cdots,l_{k},j)}\bigg).
		\end{align*}
		First, we prove that each $A_{l_j}$ is polynomially convex. Without loss of generality, it is enough to show that $A_{l_{1}}$ is polynomially convex. We show this by induction principle. Given that each $\bigg(\cup_{I}X_{I}^{(l_1,\cdots,l_{k},j)}\bigg)$ is polynomially convex.
		Assume that 	
		\begin{align*}
			K_{1}:=\cup_{j=1}^{t}	\bigg(\cup_{I}X_{I}^{(l_1,\cdots,l_{k},j)}\bigg), ~~t<l_1-1
		\end{align*}
		is polynomially convex. We need to show that 
		\begin{align*}
			\cup_{j=1}^{t+1}\bigg(\cup_{I}X_{I}^{(l_1,\cdots,l_{k},j)}\bigg)=K_{1}\cup X_{I}^{(l_1,\cdots,l_{k},t+1)}=:K_{1}\cup K_{2}
		\end{align*}
		is polynomially convex.	For this, we consider the polynomial
		\begin{align*}
			p_{t}(z,w)=\prod^{t}_{\substack{j=1\\
					j\notin\{l_1,\cdots,l_{k}\}	\\
			}}
			z_{j}w_{j}.
		\end{align*}
		\noindent Then we have the following:
		\begin{itemize}
			\item $p_{t}(K_{1})=\{0\}.$ 
			\item $p_{t}(K_{2}\setminus K_{1})\not 
			=\{0\}:$ using \Cref{L:KLp1M}, assertion (iii), we get that for $(z,w)\in K_{2}\setminus K_{1},$ $p_{t}(z,w)=0$ implies
			$\displaystyle \prod^{t}_{\substack{j=1\\
					j\notin\{l_1,\cdots,l_{k}\}	\\
			}}
			z_{j}=0.$ Since $(z,w) \notin K_{1},$ this is not possible. Therefore, $p_{t}(K_{2}\setminus E_{1})\not 
			=\{0\}$ and  $p_{t}(K_{1})\cap p_{t}(K_{2})=\{0\}.$
			\item by \Cref{L:KLp1M}, assertion (iii), we get that $p_{t}^{-1}\{0\}\cap K_{1} =K_{1}$ and $p_{t}^{-1}\{0\}\cap K_{2}\subset K_{1}.$ Therefore, $p_{t}^{-1}\{0\}\cap \bigg(K_1\cup K_{2}\bigg)=K_{1}$ is polynomially convex.
		\end{itemize}
		
		\noindent Therefore, using Kallin's lemma we get that 
		$K_{1}\cup K_{2}=		\cup_{j=1}^{t+1}\bigg(\cup_{I}X_{I}^{(l_1,\cdots,l_{k},j)}\bigg)$ is polynomially convex, and hence, by induction principle, $A_{l_1}$ is polynomially convex.
		
		Now we show that
		\begin{align*}
			\cup_{j\in \{1,\cdots,n\}\setminus\{l_1,\cdots,l_{k}\}}	\bigg(\cup_{I}X_{I}^{(l_1,\cdots,l_{k},j)}\bigg)=\cup_{j=1}^{k+1}A_{l_j},
		\end{align*}
		is polynomially convex. Again, we will apply induction principle. Assume that $E_{1}:=\cup_{j=1}^{s}A_{l_j}$ is polynomially convex for $s<k+1$. We need to show that
		\begin{align*}
			\cup_{j=1}^{s+1}A_{l_j}=E_{1}\cup A_{l_{s+1}}=:E_{1}\cup E_{2}
		\end{align*} 
		is polynomially convex. 
		\noindent	We consider the polynomial
		
		\begin{align*}
			p_{s}(z,w)=	\prod^{s}_{\substack{j=1\\
					j\notin\{l_1,\cdots,l_{k}\}	\\
			}}
			z_{j}w_{j}.
		\end{align*}
		Then we have the following:
		\begin{itemize}
			\item $p_{s}(E_{1})=\{0\}.$
			\item $p_{s}(E_{2}\setminus E_{1})\not 
			=\{0\}:$ using \Cref{L:KLp1M}, assertion (iii), we get that for $(z,w)\in E_{2}\setminus E_{1},$ $p_{s}(z,w)=0$ implies
			$\displaystyle \prod^{s}_{\substack{j=1\\
					j\notin\{l_1,\cdots,l_{k}\}	\\
			}}
			z_{j}=0.$ Since $(z,w) \notin E_{1},$ this is not possible. Therefore, $p_{s}(E_{2}\setminus E_{1})\not 
			=\{0\}$ and  $p_{s}(E_{1})\cap p_{s}(E_{2})=\{0\}.$
			\item applying \Cref{L:KLp1M}, assertion (iii), we get that $p_{s}^{-1}\{0\}\cap (E_{1}\cup E_{2})=E_{1},$ which is polynomially convex.
		\end{itemize}	
		Again, by Kallin's lemma, we conclude that $\cup_{j=1}^{s+1}A_{l_j}$ is polynomially convex. Therefore, by induction principle, we obtain that $\cup_{j=1}^{k+1}A_{l_j}$ is polynomially convex.
	\end{proof}	
	
	\begin{lemma}\label{L:Gnrl_stpM}
		Assume that $\{l_1,\cdots,l_{k}\}\subset\{1,\cdots,n\}$ and for each $j\in\{1,\cdots,n\}\setminus\{l_1,\cdots,l_{k}\},$ $\cup_{I}X_{I}^{(l_1,\cdots,l_{k},j)}$ is polynomially convex.
		Then $\cup_{I}X_{I}^{(l_1,\cdots,l_{k})}$ is polynomially convex.
	\end{lemma}		
	\begin{proof}
		Since $\cup_{I}X_{I}^{(l_1,\cdots,l_{k},j)}$ is polynomially convex, by \Cref{L:Mk-th term}, $\cup_{j\notin\{l_1,\cdots,l_{k}\}}^{n}	\left(\cup_{I}X_{I}^{(l_1,\cdots,l_{k},j)}\right)$ is polynomially convex.
		\noindent We consider the polynomial
		\begin{align*}
			p_{1}(z,w)=	\prod^{n}_{\substack{j=1\\
					j\notin\{l_1,\cdots,l_{k}\}	\\
			}}
			z_{j}w_{j}.
		\end{align*}
		\noindent Therefore
		\begin{itemize}
			\item By \Cref{L:KLp1M}, each $p_{1}\bigg(X_{I}^{(l_1,\cdots,l_{k})}\bigg)\subset L_{I}^{(l_1,\cdots,l_k)}$  with $L_{I}^{(l_1,\cdots,l_k)}\cap L_{J}^{(l_1,\cdots,l_k)}=\{0\},$ where each $L_{I}^{(l_1,\cdots,l_k)}$ is the half real line starting from the origin with argument $\left({{\displaystyle2\pi\sum^{2n}_{\substack{j=1\\
							j\notin\{i_1,\cdots,i_{k}\}	\\
					}}
					\frac{t_j}{m_j}}}\right).$
			\item By \Cref{L:KLp1M}, assertion (iii), we get that $p_{1}^{-1}\{0\}\cap\bigg(\cup_{I}X_{I}^{(l_1,\cdots,l_{k})}\bigg)=	\cup_{j\notin\{l_1,\cdots,l_{k}\}}^{n}	\bigg(\cup_{I}X_{I}^{(l_1,\cdots,l_{k},j)}\bigg).$
		\end{itemize}
		From above we can say that $p_{1}^{-1}\{0\}\cap\bigg(\cup_{I}X_{I}^{(l_1,\cdots,l_{k})}\bigg)$ is polynomially convex. Therefore, by Kallin's lemma, we infer that 
		$\left(\cup_{I}X_{I}^{(l_1,\cdots,l_{k})}\right)$ is polynomially convex.		
	\end{proof}
	We now begin the proof of \Cref{T:Minsker_Scv}.
	\begin{proof}[Proof of Theorem~\ref{T:Minsker_Scv}]
		Since $X_{I}$ is the image of a compact subset of $\mathbb{R}^{2n}\subset \cplx^{2n}$ under an invertible $\cplx$-linear map, 	$\poly(X_{I})=\smoo(X_{I})$  for all $I\in \{0,1,\cdots,m_1-1\}\times \cdots\times \{0,1,\cdots,m_{2n}-1\}.$ 
		
		\noindent Next, we wish to show that $\poly(\Phi^{-1}(X))= \smoo(\Phi^{-1}(X)).$\\
		
		\smallskip 		
		\noindent We consider the polynomial
		\begin{align*}
			p(z,w)=	\prod^{n}_{\substack{j=1
			}}
			z_{j}w_{j}.
		\end{align*}
		\noindent For $\{i_1,\cdots,i_k\}=\emptyset,$ in view of \Cref{L:KLp1M}, we can say that:
		\begin{itemize}
			\item For each $I=(t_1,\cdots,t_{2n})\in \{0,1,\cdots,m_1-1\}\times \cdots\times \{0,1,\cdots,m_n-1\},$ $p(X_{I})\subset L_{I},$ where $L_{I}$ is a half-line starting from the origin with argument $\left({{\displaystyle 2\pi\sum^{2n}_{\substack{j=1
					}}
					\frac{t_j}{m_j}}}\right).$
			\item $L_{I}\cap L_{J}=\{0\}~\text{ for all } {I\ne J}.$
		\end{itemize}
		
		\noindent In view of \Cref{L:Gnrl_stpM}, and by application of induction principle on $n,$ we can say that
		
		\begin{align*}
			p^{-1}\{0\}\cap\bigg(\cup_{I}X_{I}\bigg)=\cup^{n}_{j=1}\left(\cup_{I}X_{I}^{(j)}\right)
		\end{align*} 
		is polynomially convex.
		\noindent Since for each $I\in \left\{0,1,\cdots,m_1-1\right\}\times \cdots \times\left\{0,1,\cdots,m_n-1\right\},$ $\poly(X_{I})= \smoo(X_{I}),$ applying \Cref{R:Sto_apprx.} we conclude that 
		\begin{displaymath}
			\poly\left(\cup_{I}X_I\right)= \smoo\left(\cup_{I}X_I\right),~ \text{ that is } \poly\left(\Phi^{-1}(X)\right)=\smoo\left(\Phi^{-1}(X)\right).
		\end{displaymath}
		So, by \Cref{R:Sto_propr}, we get that
		\begin{displaymath}
			\poly(X)= \smoo(X).
		\end{displaymath}		
		
		\noindent Clearly, 
		$z^{m_{1}}_{1},\cdots, z^{m_{n}}_{n}, \bar{z_1}^{m_{n+1}} ,\cdots , \bar{z_{n}}^{m_{2n}}$ separates points on $\overline{\mathbb{D}}^{n}$ because $\gcd(m_i,m_j)=1\text{ for all } i\ne j.$
		\noindent Therefore,
		\begin{align*}      
			[z^{m_1}_{1},\cdots, z^{m_n}_{n}, \bar{z_1}^{m_{n+1}},\cdots , \bar{z_{n}}^{m_{2n}}; \overline{\mathbb{D}}^{n}] = C(\overline{\mathbb{D}}^{n}).
		\end{align*}
		
	\end{proof}

	
	\section{Proof of \Cref{T:Main Result1}}\label{S:Main Result1}
	The structure of the proof of \Cref{T:Main Result1} is similar to that of \Cref{T:Main Result2}. Before going into the proof of \Cref{T:Main Result1}, we need some preparations, including some lemmas as in the proof of \Cref{T:Main Result2}.
	\par 
	We set	
	\begin{align*}
		X:=\{(z^{m_{1}}_{1},\cdots, z^{m_{n}}_{n}, \bar{z_1}^{m_{n+1}} + R_{1}(z),\cdots , \bar{z_{n}}^{m_{2n}} + R_{n}(z)): z\in D(\delta_{0})\},
	\end{align*} and we define $\Phi:\mathbb{C}^{2n}\to \mathbb{C}^{2n}$ by 
	\begin{align*}
		\Phi\left(z_1,\cdots, z_{n},w_{1},\cdots,w_{n}\right)=\left(z^{m_{1}}_1,\cdots, z^{m_{n}}_{n},w_{1},\cdots,w_{n}\right).~~~\Phi \text{ is a proper holomorphic map.}
	\end{align*}
	\noindent We have,
	$\Phi^{-1}(X)=\cup_{I}X_{I}, \text{ where }$ $I\in \{0,1,\cdots,m_1-1\}\times \cdots\times \{0,1,\cdots,m_n-1\}$ and for each $I=(t_{1},t_{2},\cdots ,t_{n}),$
	\allowdisplaybreaks
	\begin{align*}
		X_{I}:=\left\{\left(e^{\frac{2i\pi t_{1}}{m_1} }z_{1},\cdots, e^{\frac{2i\pi t_{n}}{m_n}}z_{n}, \bar{z_1}^{m_{n+1}} + R_{1}(z),\cdots , \bar{z_{n}}^{m_{2n}} + R_{n}(z)\right): z\in D(\delta_{0}) \right\}.
	\end{align*}
	We will denote $X_{(0,\cdots,0)}$ by $X_{0}.$		
	For $\{i_1,\cdots,i_k\}\subseteq\{1,\cdots,n\}$ and $X_{I}$ as above, we denote
	\begin{align*}
		X^{(i_1,\cdots,i_k)}_{I}:=\bigg\{\bigg(e^{\frac{2i\pi t_{1}}{m_1} }z_{1},\cdots, e^{\frac{2i\pi t_{n}}{m_n}}z_{n}, \bar{z_1}^{m_{n+1}}& + R_{1}(z),\cdots , \bar{z_{n}}^{m_{2n}} + R_{n}(z)\bigg):\\ &z\in D(\delta_{0}), z_{i_l}=0, l=1,\cdots,k\bigg\}.
	\end{align*}
	
	We consider the polynomial
	\begin{align*}
		p_{k}(z,w))=	\prod^{n}_{\substack{j=1\\
				j\notin\{i_1,\cdots,i_{k}\}	\\
		}}
		z_{j}^{m_{n+j}}w_{j},
	\end{align*}	
	where $z=(z_1,\cdots,z_{n}),w=(w_1,\cdots,w_{n})\in \cplx^{n}.$

	\begin{lemma}\label{L:KLp1}
		$p_{k}(X^{(i_1,\cdots,i_k)}_{I})\subset \omega^{(i_1,\cdots,i_k)}_{I}$, where $\omega^{(i_1,\cdots,i_k)}_{I}$ is a closed sector in the complex plane with vertex at the origin and $L^{(i_1,\cdots,i_k)}_{I}$  as the angular bisector with argument $\left({\displaystyle 2\pi\sum^{n}_{\substack{j=1\\
					j\notin\{i_1,\cdots,i_{k}\}	\\
			}}
			\frac{t_jm_{n+j}}{m_j}}\right).$
		Furthermore, $\omega^{(i_1,\cdots,i_k)}_{I}\cap\omega^{(i_1,\cdots,i_k)}_{J}=\{0\}~~\text{ for distinct } X^{(i_1,\cdots,i_k)}_{I}$ and $X^{(i_1,\cdots,i_k)}_{J}.$
		
	\end{lemma}	
	\begin{proof}
		Let $(z,w)\in X^{(i_1,\cdots,i_k)}_{0}.$
		\noindent We now do some computations here:
		\begin{align}\label{E:Image of G}
			\notag p_{k}(z,w)&=\prod^{n}_{\substack{j=1\\
					j\notin\{i_1,\cdots,i_{k}\}	\\
			}}
			\bigg(	|z_{j}|^{2m_{n+j}}+	z_{j}^{m_{n+j}}R_{j} \bigg)\\\notag
			&=\prod^{n}_{\substack{j=1\\
					j\notin\{i_1,\cdots,i_{k}\}	\\
			}}
			|z_{j}|^{2m_{n+j}}+\bigg [\sum_{1\le l_1\le n}\al_{(l_1)}(z)+\sum_{1\le l_1\ne l_2\le n}\al_{(l_1,l_2)}(z)+\cdots\\
			& \quad +\sum_{1\le l_1\ne l_2\ne \cdots\ne l_{n-k-1}\le n} \al_{(l_1,\cdots,l_{n-k-1})}(z)
			+\al_{(l_1,\cdots,l_{n-k})}(z)\bigg],
		\end{align}
		\noindent where $\{l_{1},\cdots,l_{r}\}\subset\{1,\cdots,n\}\setminus\{i_1,\cdots,i_{k}\}$ and
		
		\begin{align*}
			\al_{(l_1,\cdots,l_r)}(z):=\bigg(	\prod^{n}_{\substack{j=1\\
					j\notin\{i_1,\cdots,i_{k}\}	\\
					j\notin\{l_1,\cdots,l_r\}
			}}
			|z_{j}|^{2m_{n+j}} \bigg)
			\prod_{\substack{
					j\in\{l_1,\cdots,l_{r}\}	\\
			}}
			\bigg({z_{j}}^{m_{n+j}} R_{j}\bigg).
		\end{align*}	
		
		\noindent	Since $R_{j}(z)\sim o(|z_{j}|^{m_{n+j}}) \text{ as } z_{j}\to 0$ for $j=1,\cdots,n;$ choose $\eps'>0,$ then there exists $\delta_{j}>0$ such that 
		\begin{align}\label{E:Order of R_j}
			|R_{j}|\le\eps'|z_j|^{m_{n+j}} \text{ whenever } |z_{j}|\le \delta_{j}~~ \forall {j}=1,2,\cdots,n. 
		\end{align}
		Taking $\delta'=\min_{1\le j\le n}\{\delta_j\},$ we obtain from (\ref{E:Order of R_j}) that 
		
		\allowdisplaybreaks
		\begin{align}\label{E:estmate_rth trm}
			\notag	\left|\sum_{1\le l_1\ne \cdots\ne l_r\le n}
			\al_{(l_1,\cdots,l_r)}(z)\right|&\le \left|\sum_{1\le l_1\ne \cdots\ne l_r\le n} \bigg(	\prod^{n}_{\substack{j=1\\
					j\notin\{i_1,\cdots,i_{k}\}	\\
					j\notin\{l_1,\cdots,l_r\}
			}}
			|z_{j}|^{2m_{n+j}} \bigg)
			\prod_{\substack{
					j\in\{l_1,\cdots,l_{r}\}\\
			}}
			\bigg({z_{j}}^{m_{n+j}} R_{j}\bigg)\right|\\\notag
			&\le \sum_{1\le l_1\ne \cdots\ne l_r\le n} \bigg(	\prod^{n}_{\substack{j=1\\
					j\notin\{i_1,\cdots,i_{k}\}	\\
					j\notin\{l_1,\cdots,l_r\}
			}}
			|z_{j}|^{2m_{n+j}} \bigg)
			\prod_{\substack{
					j\in\{l_1,\cdots,l_{r}\}\\
			}}
			\left |\bigg({z_{j}}^{m_{n+j}} R_{j}\bigg)\right|\\
			&\le{n-k\choose r}{\eps'}^r	\prod^{n}_{\substack{j=1\\
					j\notin\{i_1,\cdots,i_{k}\}	\\
			}}
			|z_{j}|^{2m_{n+j}}~~\forall z\in \overline{D(\delta')}. 
		\end{align}
		\noindent	Writing $p_{k}(z)=u+iv,$ we get from (\ref{E:Image of G}) that for $(z,w)	\in X^{(i_1,\cdots,i_k)}_{0},$ 
		\begin{align*}
			|v|&=|\imag(p_{k}(z,w))|\\
			&=\bigg|\imag\bigg [\sum_{1\le l_1\le n}\al_{(l_1)}(z)+\sum_{1\le l_1\ne l_2\le n}\al_{(l_1,l_2)}(z)+\cdots\\
			& \quad +\sum_{1\le l_1\ne l_2\ne \cdots\ne l_{n-k-1}\le n} \al_{(l_1,\cdots,l_{n-k-1})}(z)
			+\al_{(l_1,\cdots ,l_{n-k})}(z)\bigg]\bigg|\\
			&\le \bigg|\sum_{1\le l_1\le n}\al_{(l_1)}(z)\bigg|+\bigg|\sum_{1\le l_1\ne l_2\le n}\al_{(l_1,l_2)}(z)\bigg|+\cdots+\bigg|\sum_{\substack{1\le l_1\ne l_2\ne \\\cdots\ne l_{n-k-1}\le n}}\al_{(l_1,\cdots ,l_{n-k-1})}(z)\bigg|+\bigg|\al_{(l_1,\cdots,l_{n})}(z)\bigg|\\
			&\le\sum_{r=1}^{n-k}{n-k\choose r}{\eps'}^r	\prod^{n}_{\substack{j=1\\
					j\notin\{i_1,\cdots,i_{k}\}	\\
			}}
			|z_{j}|^{2m_{n+j}}.
		\end{align*}
		\noindent We denote $\eps:= \displaystyle\sum_{r=1}^{n-k}{n-k\choose r}{\eps'}^r.$
		\noindent Therefore,
		\begin{align}\label{E:Imaginary part of p}
			|v|\le\eps	\prod^{n}_{\substack{j=1\\
					j\notin\{i_1,\cdots,i_{k}\}	\\
			}}
			|z_{j}|^{2m_{n+j}}~~\forall z\in \overline{D(\delta')}.
		\end{align}	
		
		\noindent Similarly, for any $(z,w)\in X^{(i_1,\cdots,i_k)}_{0},$ we get from (\ref{E:Image of G}) that the real part of $p_k$ is 	
		
		\begin{align}
			\notag u=\prod^{n}_{\substack{j=1\\
					j\notin\{i_1,\cdots,i_{k}\}	\\
			}}
			|z_{j}|^{2m_{n+j}}+&\rl\bigg [\sum_{1\le l_1\le n}\al_{l_1}(z)+\sum_{1\le l_1\ne l_2\le n}\al_{(l_1,l_2)}(z)+\cdots\\&
			\quad +\sum_{1\le l_1\ne l_2\ne \cdots\ne l_{n-1}\le n} \al_{(l_1,l_2,\cdots l_{n-k-1})}(z)
			+\al_{(l_1,l_2,\cdots ,l_{n-k})}(z)\bigg].
		\end{align}
		
		\noindent We compute:
		\begin{align*}
			&\rl\bigg [\sum_{1\le l_1\le n}\al_{(l_1)}(z)+\sum_{1\le l_1\ne l_2\le n}\al_{(l_1,l_2)}(z)+\cdots\\
			& \quad +\sum_{1\le l_1\ne l_2\ne \cdots\ne l_{n-k-1}\le n} \al_{(l_1,l_2,\cdots,l_{n-k-1})}(z)
			+\al_{(l_1,l_2,\cdots,l_{n-k})}(z)\bigg]\\	
			&\le \bigg|\sum_{1\le l_1\le n}\al_{(l_1)}(z)\bigg|+\bigg|\sum_{1\le l_1\ne l_2\le n}\al_{(l_1,l_2)}(z)\bigg|+\cdots+\bigg|\sum_{\substack{1\le l_1\ne 	l_2\ne \\\cdots\ne l_{n-k-1}}\le n}\al_{(l_1,\cdots,l_{n-k-1})}(z)\bigg|+\bigg|\al_{(l_1,\cdots,l_{n-k})}(z)\bigg|\\
			&\le\sum_{r=1}^{n-k}{n-k\choose r}{\eps'}^r	\prod^{n}_{\substack{j=1\\
					j\notin\{i_1,\cdots,i_{k}\}	\\
			}}
			|z_{j}|^{2m_{n+j}}	
			=\eps	\prod^{n}_{\substack{j=1\\
					j\notin\{i_1,\cdots,i_{k}\}	\\
			}}
			|z_{j}|^{2m_{n+j}}.
		\end{align*}
		
		\noindent Therefore 
		\allowdisplaybreaks
		\begin{align*}
			-\eps	\prod^{n}_{\substack{j=1\\
					j\notin\{i_1,\cdots,i_{k}\}	\\
			}}
			|z_{j}|^{2m_{n+j}}\le&\rl\bigg [\sum_{1\le l_1\le n}\al_{(l_1)}(z)+\sum_{1\le l_1\ne l_2\le n}\al_{(l_1,l_2)}(z)+\cdots\\
			& \quad +\sum_{1\le l_1\ne l_2\ne \cdots\ne l_{n-k-1}\le n} \al_{(l_1,l_2,\cdots,l_{n-k-1})}(z)
			+\al_{(l_1,l_2,\cdots,l_{n-k})}(z)\bigg]\\
			&\le \eps	\prod^{n}_{\substack{j=1\\
					j\notin\{i_1,\cdots,i_{k}\}	\\
			}}
			|z_{j}|^{2m_{n+j}}.
		\end{align*}	
		\noindent	Hence 
		\begin{align}\label{E: Real part of p}
			(1-\eps) 	\prod^{n}_{\substack{j=1\\
					j\notin\{i_1,\cdots,i_{k}\}	\\
			}}
			|z_{j}|^{2m_{n+j}}\le u \le (1+\eps)	\prod^{n}_{\substack{j=1\\
					j\notin\{i_1,\cdots,i_{k}\}	\\
			}}
			|z_{j}|^{2m_{n+j}}.
		\end{align}
		\noindent		We choose $\eps'$ in such a way such that $\eps< 1$. In view of (\ref{E:Imaginary part of p}) and (\ref{E: Real part of p}), we obtain that
		
		\begin{align}
			|v|\le \eps\prod^{n}_{\substack{j=1\\
					j\notin\{i_1,\cdots,i_{k}\}	\\
			}}
			|z_{j}|^{2m_{n+j}}\le \left(\frac{\eps}{1-\eps}\right)u, u\ge 0, \forall z\in\overline{D(\delta')}, \text{ by shrinking } \delta' \text{ if required}.
		\end{align}
		\noindent		Above inequalities says that 
		\begin{align}\label{E:image of X0}
			p_{k}(X^{(i_1,\cdots,i_k)}_{0})\subset\left\{u+iv\in \cplx: |v|\le \left(\frac{\eps}{1-\eps}\right)u,~~ u\ge 0\right\}.
		\end{align}
		
		\noindent We set $\tht':=\min_{I\ne J}\left\{\text { angle between the lines } 	L^{(i_1,\cdots,i_k)}_{I} \text{ and }  L^{(i_1,\cdots,i_k)}_{J}\right\},$ and we take $\tht<\min\{\tht',\frac{\pi}{2}\}$ where, $L^{(i_1,\cdots,i_k)}_{I}$ is the half line through the origin with argument ${\left({\displaystyle2\pi\sum^{n}_{\substack{j=1\\
						j\notin\{i_1,\cdots,i_{k}\}	\\
				}}
				\frac{t_jm_{n+j}}{m_j}}\right)}$ for $I=(t_1,\cdots,t_n)$ and 
		$L^{(i_1,\cdots,i_k)}_{J}$ is the half line through the origin with argument ${\left({\displaystyle2\pi\sum^{n}_{\substack{j=1\\
						j\notin\{i_1,\cdots,i_{k}\}	\\
				}}
				\frac{s_j m_{n+j}}{m_j}}\right)}$
		for $J=(s_1,\cdots,s_n).$
		
		\noindent Again, we shrink $\eps'$ further so that $$\eps<\frac{\tan\left(\frac{\tht}{2}\right)}{1+\tan\left(\frac{\tht}{2}\right)}.$$ 
		This implies $$\left(\frac{\eps}{1-\eps}\right)<\tan\left(\frac{\tht}{2}\right).$$
		The expression (\ref{E:image of X0}) says that $	p_{k}(X^{(i_1,\cdots,i_k)}_{0})$ lies in the closed angular sector $\omega^{(i_1,\cdots,i_k)}_{0}$ with vertex at $0,$ positive real axis as the angular bisector and has an vertex-angle $2\phi,$ where 
		
		\begin{align*}
			\tan(\phi)=\left(\frac{\eps}{1-\eps}\right)<\tan\left(\frac{\tht}{2}\right)
		\end{align*}
		i.e $\omega^{(i_1,\cdots,i_k)}_{0}$  is an angular sector with vertex at the origin, positive real axis as the angular bisector and has vertex-angle $2\phi<\tht.$  
		
		\noindent Let 	$I=(t_{1},t_{2},\cdots ,t_{n})$ be an arbitrary element of $ \{0,1,\cdots,m_1-1\}\times \{0,1,\cdots,m_2-1\}\times \cdots\times \{0,1,\cdots,m_n-1\}.$
		Since $(z,w)\in X_I,$ we get that
		\begin{align*}
			p_{k}(z,w)
			&=	e^{i\alpha}
			\prod^{n}_{\substack{j=1\\
					j\notin\{i_1,\cdots,i_{k}\}	\\
			}}
			\bigg(	|z_{j}|^{2m_{n+j}}+	z_{j}^{m_{n+j}}R_{j} \bigg), \text{ where } \alpha={\bigg({\displaystyle2\pi 	\sum^{n}_{\substack{j=1\\
							j\notin\{i_1,\cdots,i_{k}\}	\\
					}}
					\frac{t_jm_{n+j}}{m_j}}\bigg)}.
		\end{align*}

		\noindent There exist a closed sector $\omega^{(i_1,\cdots,i_k)}_{I}$ with  vertex at the origin, $L^{(i_1,\cdots,i_k)}_{I}$  as the angular bisector with argument ${\bigg({\displaystyle2\pi\sum^{n}_{\substack{j=1\\
						j\notin\{i_1,\cdots,i_{k}\}	\\
				}}
				\frac{t_jm_{n+j}}{m_j}}\bigg)}$ and has vertex-angle $<\tht.$ Clearly, $L^{(i_1,\cdots,i_k)}_{I}$ is obtained by rotation of the positive real axis at an angle  ${\bigg({\displaystyle2\pi\sum^{n}_{\substack{j=1\\
						j\notin\{i_1,\cdots,i_{k}\}	\\
				}}
				\frac{t_jm_{n+j}}{m_j}}\bigg)}.$ By putting $\alpha_{j}=m_{n+j}~~~ \forall {j}\in \{1,\cdots,n\}\text{ and } \alpha_{j}=0~~\forall {j}\in\{n+1,\cdots,2n\}$ in \Cref{L:Argumnt}, we get that
		\begin{align}\label{E:Omega_IJ}
			\omega^{(i_1,\cdots,i_k)}_{I}\cap\omega^{(i_1,\cdots,i_k)}_{J}=\{0\}~~\text{ for distinct } X^{(i_1,\cdots,i_k)}_{I} \text{ and } X^{(i_1,\cdots,i_k)}_{J}.
		\end{align}
		For $\{i_1,\cdots,i_k\}=\emptyset,$ then we denote $\omega^{(i_1,\cdots,i_k)}_{I}$ by $\omega_{I}.$	
		
	\end{proof}

	\begin{lemma}\label{L:KLp2}
		Let $p_{k}$ and $X^{(i_1,\cdots,i_k)}_{I}$ be as above. Then
		\begin{align*}
			p_{k}^{-1}\{0\}\cap X^{(i_1,\cdots,i_k)}_{I}=\bigg\{\bigg(e^{\frac{2i\pi t_{1}}{m_1} }z_{1},\cdots,& e^{\frac{2i\pi t_{n}}{m_n}}z_{n}, \bar{z_1}^{m_{n+1}} + R_{1}(z),\cdots , \bar{z_{n}}^{m_{2n}} + R_{n}(z)\bigg):\\ &z\in \overline{D(\delta')}, z_{i_l}=0, l=1,\cdots,k, \prod^{n}_{\substack{j=1\\
					j\notin\{i_1,\cdots,i_{k}\}	\\
			}}
			z_{j} =0\bigg\}.
		\end{align*}
	\end{lemma}	
	
	\begin{proof}
		\noindent Take $w=(w_1,\cdots,w_n,w_{n+1},\cdots,w_{2n})\in p_k^{-1}\{0\}\cap X^{(i_1,\cdots,i_k)}_{I}.$ Then 
		\begin{align*}
			w_{j}&=e^{\frac{2i\pi t_{j}}{m_j} }z_{j}~~ \text{ and }\\
			w_{n+j}&=\bar{z_j}^{m_{n+j}} + R_{j}(z)~~\text{ for }  j=1,2,\cdots,n;
		\end{align*}
		Since $w\in p_k^{-1}\{0\},$ 
		\begin{align*}
			\prod^{n}_{\substack{j=1\\
					j\notin\{i_1,\cdots,i_{k}\}	\\
			}}
			e^{\frac{2\pi i t_{j}m_{n+j}}{m_{j}}}	z_{j}^{m_{n+j}} \bigg(\bar{z_{j}}^{m_{n+j}} + R_{j}\bigg)=0.
		\end{align*}
		This implies 
		\begin{align}\label{E:Img_XI=0}
			e^{i\alpha}
			\bigg(	\prod^{n}_{\substack{j=1\\
					j\notin\{i_1,\cdots,i_{k}\}	\\
			}}
			z_{j}^{m_{n+j}} \bigg)
			\prod^{n}_{\substack{j=1\\
					j\notin\{i_1,\cdots,i_{k}\}	\\
			}}
			\bigg(\bar{z_{j}}^{m_{n+j}} + R_{j}\bigg)=0.
		\end{align}

		\noindent On $\overline{D(\delta')},$ $|R_j|\le \eps'|z_{j}|^{m_{n+j}},$ for some $\eps'<1,$ Therefore, we can say that 
		\begin{align*}
			\prod^{n}_{\substack{j=1\\
					j\notin\{i_1,\cdots,i_{k}\}	\\
			}}
			\bigg(\bar{z_{j}}^{m_{n+j}} + R_{j}\bigg)\ne 0.
		\end{align*}
		Hence, form (\ref{E:Img_XI=0}), we have
		\begin{align*}
			\prod^{n}_{\substack{j=1\\
					j\notin\{i_1,\cdots,i_{k}\}	\\
			}}
			z_{j}^{m_{n+j}} =0 ~~\text{ i.e } 	\prod^{n}_{\substack{j=1\\
					j\notin\{i_1,\cdots,i_{k}\}	\\
			}}
			z_{j} =0.
		\end{align*}
	\end{proof}		
	
	\begin{lemma}\label{L:k-th term}
		Assume that $\{l_1,\cdots,l_{k}\}\subset\{1,2,\cdots,n\}$ and $\cup_{I}X_{I}^{(l_1,\cdots,l_{k},j)}$ is polynomially convex for each $j\in\{1,\cdots,n\}\setminus\{l_1,\cdots,l_{k}\}.$
		Then 
		\begin{align*}
			\cup_{j\in \{1,\cdots,n\}\setminus\{l_1,\cdots,l_{k}\}}	\bigg(\cup_{I}X_{I}^{(l_1,\cdots,l_{k},j)}\bigg)
		\end{align*}	
		is polynomially convex.	
	\end{lemma}
	
	\begin{proof}
		Without loss of generality, we assume that $1\le l_1<l_2\cdots<l_k\le n.$ Clearly, 
		\begin{align*}
			\cup_{j\in \{1,\cdots,n\}\setminus\{l_1,\cdots,l_{k}\}}	\bigg(\cup_{I}X_{I}^{(l_1,\cdots,l_{k},j)}\bigg)=\cup_{j=1}^{k+1}A_{l_j},
		\end{align*}
		where 
		\begin{align*}
			A_{l_{1}}:&=\cup_{j=1}^{l_1-1}	\bigg(\cup_{I}X_{I}^{(l_1,\cdots,l_{k},j)}\bigg);\\
			A_{l_{r+1}}:&=	\cup_{j=l_{r}+1}^{l_{r+1}-1}	\bigg(\cup_{I}X_{I}^{(l_1,\cdots,l_{k},j)}\bigg)~\text{ for }r=1,\cdots,k-1;\text{ and }\\
			A_{l_{k+1}}:&=	\cup_{j=l_{k+1}}^{n}	\bigg(\cup_{I}X_{I}^{(l_1,\cdots,l_{k},j)}\bigg).
		\end{align*}
		First, we prove that each $A_{l_j}$ is polynomially convex. Without loss of generality, it is enough to show that 
		\begin{align*}
			A_{l_{1}}:&=	\cup_{j=1}^{l_1-1}	\bigg(\cup_{I}X_{I}^{(l_1,\cdots,l_{k},j)}\bigg)
		\end{align*}
		is polynomially convex. We will now apply induction principle to show $A_{l_1}$ is polynomially convex. Assume that $K_{1}:=\cup_{j=1}^{t}	\bigg(\cup_{I}X_{I}^{(l_1,\cdots,l_{k},j)}\bigg)$ is polynomially convex for $t<l_1-1$. 
		We need to show that 
		\begin{align*}
			\cup_{j=1}^{t+1}\bigg(\cup_{I}X_{I}^{(l_1,\cdots,l_{k},j)}\bigg)=K_{1}\cup X_{I}^{(l_1,\cdots,l_{k},t+1)}=:K_{1}\cup K_{2}
		\end{align*}
		is polynomially convex.	For this, we consider the polynomial
		\begin{align*}
			p_{t}(z,w)=\prod^{t}_{\substack{j=1\\
					j\notin\{l_1,\cdots,l_{k}\}	\\
			}}
			z_{j}^{m_{n+j}}w_{j}.
		\end{align*}
		\noindent Then we have the following:
		\begin{itemize}
			\item $p_{t}(K_{1})=\{0\}.$ 
			\item $p_{t}(K_{2}\setminus K_{1})\not 
			=\{0\}:$ using \Cref{L:KLp2}, we get that for $(z,w)\in K_{2}\setminus K_{1},$ $p_{t}(z,w)=0$ implies
			$\displaystyle \prod^{t}_{\substack{j=1\\
					j\notin\{l_1,\cdots,l_{k}\}	\\
			}}
			z_{j}=0.$ Since $(z,w) \notin K_{1},$ this is not possible. Therefore, $p_{t}(K_{2}\setminus K_{1})\not 
			=\{0\}$ and  $p_{t}(K_{1})\cap p_{t}(K_{2})=\{0\}.$
			\item Applying \Cref{L:KLp2}, we get that $p_{t}^{-1}\{0\}\cap K_{1} =K_{1}$ and $p_{t}^{-1}\{0\}\cap K_{2}\subset K_{1}.$ Therefore, $p_{t}^{-1}\{0\}\cap \bigg(K_1\cup K_{2}\bigg)=K_{1}$ is polynomially convex.
		\end{itemize}
		
		\noindent Therefore, using Kallin's lemma we get that 
		$K_{1}\cup K_{2}=		\cup_{j=1}^{t+1}\bigg(\cup_{I}X_{I}^{(l_1,\cdots,l_{k},j)}\bigg)$ is polynomially convex, and hence, by induction principle $A_{l_1}$ is polynomially convex.
		
		Again, we will apply induction principle to show $\cup_{j=1}^{k+1}A_{l_j}$ is polynomially convex. Assume that $E_{1}:=\cup_{j=1}^{s}A_{l_j}$ is polynomially convex for $s<k+1$. We write $\cup_{j=1}^{s+1}A_{l_j}=E_{1}\cup
		A_{l_{s+1}}=:E_{1}\cup E_{2}.$
		
		\noindent We consider the polynomial
		\begin{align*}
			p_{s}(z,w)=	\prod^{s}_{\substack{j=1\\
					j\notin\{l_1,\cdots,l_{k}\}	\\
			}}
			z_{j}^{m_{n+j}}w_{j}.
		\end{align*}
		Then we have the following:
		\begin{itemize}
			\item $p_{s}(E_{1})=\{0\}.$
			\item $p_{s}(E_{2}\setminus E_{1})\not 
			=\{0\}:$ using \Cref{L:KLp2}, we get that for $(z,w)\in E_{2}\setminus E_{1},$ $p_{s}(z,w)=0$ implies
			$\displaystyle \prod^{s}_{\substack{j=1\\
					j\notin\{l_1,\cdots,l_{k}\}	\\
			}}
			z_{j}=0.$ Since $(z,w) \notin E_{1},$ this is not possible. Therefore, $p_{s}(E_{2}\setminus E_{1})\not 
			=\{0\}$ and  $p_{s}(E_{1})\cap p_{s}(E_{2})=\{0\}.$
			
			\item Applying \Cref{L:KLp2}, we get that $p_{s}^{-1}\{0\}\cap E_{1}=E_{1}$ and $p_{s}^{-1}\{0\}\cap E_{2}\subset E_{1}.$ Therefore, $p_{s}^{-1}\{0\}\cap (E_{1}\cup E_{2})=E_{1}$ is polynomially convex.
		\end{itemize}	
		\noindent Again, by Kallin's lemma, we conclude that $\cup_{j=1}^{s+1}A_{l_j}$ is polynomially convex. Therefore, by induction principle, we obtained that $\cup_{j=1}^{k+1}A_{l_j}$ is polynomially convex, that is, $\cup_{j\notin\{l_1,\cdots,l_{k}\}}	\left(\cup_{I}X_{I}^{(l_1,\cdots,l_{k},j)}\right)$ is polynomially convex.	
	\end{proof}

	\begin{proposition}\label{P:Gnrl_stp}
		Assume that $\{l_1,\cdots,l_{k}\}\subset\{1,\cdots,n\}$ and $\cup_{I}X_{I}^{(l_1,\cdots,l_{k},j)}$ is polynomially convex for each $j\in\{1,\cdots,n\}\setminus\{l_1,\cdots,l_{k}\}.$ Then $\cup_{I}X_{I}^{(l_1,\cdots,l_{k})}$ is polynomially convex.
	\end{proposition}		
	\begin{proof}
		Since $\cup_{I}X_{I}^{(l_1,\cdots,l_{k},j)}$ is polynomially convex, by \Cref{L:k-th term}, $\cup_{j\notin\{l_1,\cdots,l_{k}\}}^{n}	\bigg(\cup_{I}X_{I}^{(l_1,\cdots,l_{k},j)}\bigg)$ is polynomially convex.
		\noindent We consider the polynomial
		\begin{align*}
			p(z,w)=	\prod^{n}_{\substack{j=1\\
					j\notin\{l_1,\cdots,l_{k}\}	\\
			}}
			z_{j}^{m_{n+j}}w_{j}.
		\end{align*}
		
		\noindent Therefore
		\begin{itemize}
			\item by \Cref{L:KLp1}, $p\left(X_{I}^{(l_1,\cdots,l_{k})}\right)\subset\omega_{I}^{(l_1,\cdots,l_k)}~~~\forall I\in \{0,1,\cdots,m_1-1\}\times \{0,1,\cdots,m_2-1\}\times \cdots\times\{0,1,\cdots,m_n-1\};$
			\item Since $\omega_{I}^{(l_1,\cdots,l_k)}\cap\omega_{J}^{(l_1,\cdots,l_k)}=\{0\},$ we have $p\left(X_{I}^{(l_1,\cdots,l_{k})}\right)\cap p\left(X_{J}^{(l_1,\cdots,l_{k})}\right)\subset \{0\};$
			\item by \Cref{L:KLp2}, $p^{-1}\{0\}\cap\left(\cup_{I}X_{I}^{(l_1,\cdots,l_{k})}\right)=	\bigcup_{j\notin\{l_1,\cdots,l_{k}\}}^{n}	\left(\cup_{I}X_{I}^{(l_1,\cdots,l_{k},j)}\right),$ which is polynomially convex by \Cref{L:k-th term}.
		\end{itemize}
		Therefore, by Kallin's lemma, we infer that $\cup_{I}X_{I}^{(l_1,\cdots,l_{k})}$ is polynomially convex.		
	\end{proof}
	\smallskip 		
	\begin{proof}[Proof of \Cref{T:Main Result1}] Take $\delta:=\min\{\delta',\delta_{0}\}.$ We divide the proof of this theorem in three steps.\\
		
		\noindent  {\bf Step I:} {\bfseries\boldmath Showing	$\poly(X_{I})=\smoo(X_{I})~~~\forall_{I}\in \{0,1,\cdots,m_1-1\}\times \{0,1,\cdots,m_2-1\}\times \cdots\times \{0,1,\cdots,m_n-1\}.$}
		\par We have
		\begin{align*}
			X_{0}&=\{(z_{1},\cdots, z_{n}, \bar{z_1}^{m_{n+1}} + R_{1}(z),\cdots , \bar{z_{n}}^{m_{2n}} + R_{n}(z)): z\in \overline{D(\delta)} \}\\
			&=Gr_{\overline{D(\delta)}}(\bar{z_1}^{m_{n+1}} + R_{1}(z),\cdots , \bar{z_{n}}^{m_{2n}} + R_{n}(z)).
		\end{align*}
		
		\sloppy For each $I=(t_1,\cdots,t_{n}),$ we define a map $\Psi_{I}:\mathbb{C}^{2n}\to \mathbb{C}^{2n}$ by
		\begin{align*}
			\Psi_{I}(z_{1},\cdots,z_{n},w_{1},w_{2},\cdots,w_{n})=\left(e^{\frac{2i\pi t_{1}}{m_1} }z_{1},e^{\frac{2i\pi t_{2}}{m_2} }z_{2}\cdots, e^{\frac{2i\pi t_{n}}{m_n}}z_{n}, w_{1}, w_{2}\cdots, w_{n}\right).
		\end{align*}
		Therefore, $\Psi_{I}(X_{0})=X_{I}.$ \Cref{T:Main Result2} gives us that $\poly(X_{0})=\smoo(X_{0}).$ Since $\Psi_{I}$ is a biholomorphism, using \Cref{R:Sto_propr}, we conclude that $\poly(X_{I})=\smoo(X_{I}).$\\
		
		\medskip	
		\noindent  {\bf Step II:} {\bfseries\boldmath Showing $\poly(\Phi^{-1}(X))= \smoo(\Phi^{-1}(X)).$}\\
		\par We consider the polynomial
		\begin{align*}
			p(z,w)=	\prod^{n}_{\substack{j=1
			}}
			z_{j}^{m_{n+j}}w_{j}.
		\end{align*}
		
		\smallskip
		\noindent For $\{i_1,\cdots,i_k\}=\emptyset,$ in view of \Cref{L:KLp1}, we get that
		\begin{itemize}
			\item For each $I=(t_1,\cdots,t_{n})\in \{0,1,\cdots,m_1-1\}\times \cdots \times\{0,1,\cdots,m_n-1\},$ $p(X_{I})\subset\omega_{I},$ where $\omega_{I}$ is a closed sector in the complex plane with vertex at the origin and $L_{I}$  as the angular bisector with argument  argument ${\bigg({\displaystyle2\pi\sum^{n}_{\substack{j=1
					}}
					\frac{t_jm_{n+j}}{m_j}}\bigg)};$
			
			\item $\omega_{I}\cap\omega_{J}=\{0\}~\text{ for all } {I\ne J}.$
		\end{itemize}
		
		\noindent To apply Kallin's lemma, we need to show that $p^{-1}\{0\}\cap\bigg(\cup_{I}X_{I}\bigg)=\cup^{n}_{j=1}\left(\cup_{I}X_{I}^{(j)}\right)$ is polynomially convex. First we focus on the the polynomial convexity of $\cup_{I}X_{I}^{(j)} \text{ for } j\in{1,\cdots,n}.$	\Cref{P:Gnrl_stp} says that for each $i_{1}\in \{1,\cdots,n\},$ $\cup_{I}X_{I}^{(i_1)}$ is polynomially convex if for each $i_2\in \{1,\cdots,n\}\setminus\{i_1\},$ $\cup_{I}X_{I}^{(i_1,i_2)}$ is polynomially convex. Again by \Cref{P:Gnrl_stp}, $\cup_{I}X_{I}^{(i_1,i_2)}$ is polynomially convex if for each $i_3\in \{1,\cdots,n\}\setminus\{i_1,i_2\},$ $\cup_{I}X_{I}^{(i_1,i_2,i_3)}$ is polynomially convex. Proceeding in this way, we arrive at this situation that $\cup_{I}X_{I}^{(i_1,i_2,\cdots,i_{n-1})}$ is polynomially convex if for each $i_{n}\in\{1,\cdots,n\}\setminus\{i_1,i_2,\cdots,i_{n-1}\},$ $\cup_{I}X_{I}^{(i_1,i_2,\cdots,i_{n-1},i_n)}$ is polynomially convex. But $\cup_{I}X_{I}^{(i_1,i_2,\cdots,i_{n-1},i_n)}=\{0\}$ for all $I,$ which is obviously polynomially convex. Hence, for each $j\in \{1,\cdots,n\},$ $\cup_{I}X_{I}^{(j)}$ is polynomially convex. We will now apply induction on $n,$ to conclude that $p^{-1}\{0\}\cap\left(\cup_{I}X_{I}\right)$ is polynomially convex. Assume that $K_{1}:=\cup^{k}_{j=1}\left(\cup_{I}X_{I}^{(j)}\right)$ is polynomially convex. We need to show that $\cup^{k+1}_{j=1}\left(\cup_{I}X_{I}^{(j)}\right)=K_{1}\cup \left(\cup_{I}X_{I}^{(k+1)}\right)=:K_{1}\cup K_{2}$ is polynomially convex. For this we again apply Kallin's lemma. We consider the polynomial
		\begin{align*}
			p_1(z,w)=\prod^{k}_{\substack{j=1
			}}
			z_{j}^{m_{n+j}}w_{j}.
		\end{align*}
		Then we have the following:
		\begin{itemize}
			\item $p_1(K_{1})=\{0\}.$ 
			\item $p_1(K_{2}\setminus K_{1})\not 
			=\{0\}:$ using \Cref{L:KLp2}, we get that for $(z,w)\in K_{2}\setminus K_{1},$ $p_1(z,w)=0$ implies
			$\displaystyle \prod^{k}_{\substack{j=1\\
			}}
			z_{j}=0.$ Since $(z,w) \notin K_{1},$ this is not possible. Therefore, $p_1(K_{2}\setminus K_{1})\not 
			=\{0\}$ and  $p_1(K_{1})\cap p_1(K_{2})=\{0\}.$
			
			\item applying \Cref{L:KLp2}, we get that $p_1^{-1}\{0\}\cap K_{1}=K_{1}$  and $p_1^{-1}\{0\}\cap K_{2}\subset K_{1}.$ Hence $p_1^{-1}\{0\}\cap \left(K_{1}\cup K_{2}\right)=K_{1}$ is polynomially convex.
		\end{itemize}
		Therefore, by Kallin's lemma, $K_{1}\cup K_{2}$ is polynomially convex.	
		\noindent Again from  {\bf Step I}, we get that for each $I\in \left\{0,1,\cdots,m_1-1\right\}\times \cdots \times\left\{0,1,\cdots,m_n-1\right\},$ $\poly(X_{I})= \smoo(X_{I}).$
		
		\noindent The above informations allow us to apply \Cref{R:Sto_apprx.} to conclude that 
		\begin{displaymath}
			\poly\left(\cup_{I}X_I\right)= \smoo\left(\cup_{I}X_I\right),~ \text{ that is } \poly\left(\Phi^{-1}(X)\right)=\smoo\left(\Phi^{-1}(X)\right).
		\end{displaymath}
		\noindent Therefore, by \Cref{R:Sto_propr}, we get that
		\begin{displaymath}
			\poly(X)= \smoo(X).
		\end{displaymath}

		\noindent  {\bf Step III:} {\bfseries\boldmath Showing that ${z^{m_{1}}_{1},\cdots, z^{m_{n}}_{n}, \bar{z_1}^{m_{n+1}} + R_{1}(z),\cdots , \bar{z_{n}}^{m_{2n}} + R_{n}(z)}$ separates points on $\mathbf{\overline{D(\delta)}}.$}\\ 
		\par Let $a=(a_{1},a_{2},\cdots, a_{n}), b=(b_{1},b_{2},\cdots,b_{n})\in \mathbf{\overline{D(\delta)}}$ with $a\ne b$. Then there exists a set $\{i_1,\cdots,i_k\} \subseteq \{1,\cdots,n\}$ such that $a_{j}\ne b_{j}~~\forall j\in \{i_1,\cdots,i_k\}$ and $a_{j}= b_{j}~~\forall j\notin \{i_1,\cdots,i_k\}.$ If $a^{m_l}_{l}\ne b^{m_l}_{l} \text{ for some } l\in \{i_1,\cdots,i_k\}$ then $z^{m_{l}}_{l}$ separates $a$ and $b.$ Next, we assume that $ a^{m_j}_{j}=b^{m_j}_{j} \text{ for all } j\in \{i_1,\cdots,i_k\}.$ We now show that, for some $j\in\{i_1,\cdots,i_k\},~~\bar{z_{j}}^{m_{n+j}}+R_{j}(z)$ separates $a$ and $b.$ If possible, assume that
		\begin{align*}
			\left(\bar{z_{j}}^{m_{n+j}}+R_{j}\right)(a)=(\bar{z_{j}}^{m_{n+j}}+R_{j})(b)~~~\forall {j} \in \{i_1,\cdots,i_k\}.
		\end{align*}
		This implies 
		\begin{align}\label{E:Pt_Separte1}
			\bar{a_{j}}^{m_{n+j}}-\bar{b_{j}}^{m_{n+j}}=R_{j}(b)-R_{j}(a)~~\forall {j} \in \{i_1,\cdots,i_k\}.
		\end{align}
		Now we compute:
		\begin{align*}
			\sqrt{\sum_{j\in\{i_1,\cdots,i_k\}} |R_{j}(b)-R_{j}(a)|^2}\le|R(b)-R(a)|&\le c \left(\sum_{j=1}^n \left|b_j^{m_{n+j}}-a_j^{m_{n+j}}\right|^2\right)^{\frac{1}{2}}\\
			&=c \left(\sum_{j\in\{i_1,\cdots,i_k\}} \left|b_j^{m_{n+j}}-a_j^{m_{n+j}}\right|^2\right)^{\frac{1}{2}}.
		\end{align*}
		\noindent Therefore, using (\ref{E:Pt_Separte1}), we obtain from above that
		\begin{align}\label{E:Pt_Separte2}
			\sqrt{\sum_{j\in\{i_1,\cdots,i_k\}} \left|b_j^{m_{n+j}}-a_j^{m_{n+j}}\right|^2} 
			&\le c \left(\sum_{j\in\{i_1,\cdots,i_k\}} \left|b_j^{m_{n+j}}-a_j^{m_{n+j}}\right|^2\right)^{\frac{1}{2}}.
		\end{align}
		Since $a_{j}\ne b_{j},~~ a^{m_j}_{j}=b^{m_j}_{j} \text{ for all } j\in \{i_1,\cdots,i_k\}$ and $\gcd(m_j,m_{n+j})=1,$ we can say that 
		\begin{align*}
			{\sum_{j\in\{i_1,\cdots,i_k\}} \left|b_j^{m_{n+j}}-a_j^{m_{n+j}}\right|^2}\ne 0.
		\end{align*}
		Therefore, from (\ref{E:Pt_Separte2}) we get that $c\ge 1.$ This is a contradiction to our hypothesis.
		Therefore, 
		\begin{align*}
			[z^{m_{1}}_{1},\cdots, z^{m_{n}}_{n}, \bar{z_1}^{m_{n+1}} + R_{1}(z),\cdots , \bar{z_{n}}^{m_{2n}} + R_{n}(z); \overline{D(\delta)}] = C(\overline{D(\delta)}).
		\end{align*}	
	\end{proof}

	\noindent {\bf Acknowledgements.}
	I would like to thank Sushil Gorai for the discussion during the course of this work. This work is supported by an INSPIRE Fellowship (IF 160487) funded by DST.


\end{document}